\documentclass[hidelinks,onefignum,onetabnum]{siamonline190516}

\usepackage{lipsum}
\usepackage{amsfonts}
\usepackage{graphicx}
\usepackage{epstopdf}
\usepackage{algorithmic}
\ifpdf
  \DeclareGraphicsExtensions{.eps,.pdf,.png,.jpg}
\else
  \DeclareGraphicsExtensions{.eps}
\fi

\usepackage{enumitem}
\setlist[enumerate]{leftmargin=.5in}
\setlist[itemize]{leftmargin=.5in}

\newsiamremark{remark}{Remark}
\newsiamremark{hypothesis}{Hypothesis}
\crefname{hypothesis}{Hypothesis}{Hypotheses}
\newsiamthm{claim}{Claim}

\headers{Explicit complex time integrators for stiff problems.}{J. D. George, J. Koellermeier, S. Y. Jung, and N. M. Mangan}

\title{Explicit complex time integrators for stiff problems.}

\author{Jithin D. George\thanks{Department of Engineering Sciences and Applied Mathematics, Northwestern University, Evanston, IL.\\
(\email{jithindgeorge93@gmail.com}, \email{samyjung@outlook.com}, \email{niallmm@gmail.com}).}
\and Julian Koellermeier\thanks{Bernoulli Institute, University of Groningen, Netherlands.
(\email{j.koellermeier@rug.nl}).}
\and Samuel Y. Jung\footnotemark[1]
\and Niall M. Mangan\footnotemark[1]}

\usepackage{amsopn}

\makeatletter
\newcommand*{\addFileDependency}[1]{
  \typeout{(#1)}
  \@addtofilelist{#1}
  \IfFileExists{#1}{}{\typeout{No file #1.}}
}
\makeatother

\newcommand*{\myexternaldocument}[1]{%
    \externaldocument{#1}%
    \addFileDependency{#1.tex}%
    \addFileDependency{#1.aux}%
    	}

\ifpdf
\hypersetup{
  pdftitle={Explicit complex time integrators for stiff problems.},
  pdfauthor={J. D. George, J. Koellermeier, S. Y. Jung, and N. M. Mangan}
}
\fi

\myexternaldocument{supplement}

\usepackage{amsmath,mathtools}

\DeclareMathOperator{\sech}{sech}
\usepackage{booktabs}

\usepackage{tabularx}
\newenvironment{mat}{\left[ \begin{array}{ccccccccccccc}}{\end{array}\right]}
\newcommand\bcm{\begin{mat}}
\newcommand\ecm{\end{mat}}
\usepackage{amsfonts,epsfig}

\usepackage{forest}
\usepackage{cleveref}
\usepackage{xcolor}

\usepackage[normalem]{ulem}      
\newcommand{\complexval}[1]{\mathbf{#1}}
\newif\ifreview
\reviewtrue           

\ifreview
  \newcommand{\delete}[1]{{\color{red}#1}}  
  \newcommand{\deleteblock}[1]{{\color{red}#1}}    
\else
  
  \newcommand{\delete}[1]{}
  \newcommand{\deleteblock}[1]{}
\fi 
\forestset{
  */.style={
    delay+={append={[]},}
  },
  rooted tree/.style={
    for tree={
      grow'=90,
      parent anchor=center,
      child anchor=center,
      s sep=2.5pt,
      if level=0{
        baseline
      }{},
      delay={
        if content={*}{
          content=,
          append={[]}
        }{}
      }
    },
    before typesetting nodes={
      for tree={
        circle,
        fill,
        minimum width=3pt,
        inner sep=0pt,
        child anchor=center,
      },
    },
    before computing xy={
      for tree={
        l=5pt,
      }
    }
  }
}
\DeclareDocumentCommand\rootedtree{o}{\Forest{rooted tree [#1]}}
\begin{document}
\maketitle

\begin{abstract}
Most numerical methods for time integration use real-valued time steps. Complex time steps, however, can provide an additional degree of freedom, as we can select the magnitude of the time step in both the real and imaginary directions. We show that specific paths in the complex time plane lead to expanded stability regions, providing clear computational advantages for complex-valued systems. In particular, for systems like the Schrödinger equation, we prove that complex time integrators are uniquely optimal. Furthermore, we demonstrate that these benefits extend to certain classes of real-valued stiff systems by coupling complex time steps with the Projective Integration method.
\end{abstract}



\section{Introduction}
The complex plane has often been employed to solve real-world problems. Two direct examples are the solution of complicated real-valued integrals via contour integrals \cite{brown2009complex} and the usage of the intuition gained from poles in the complex plane in control-theory applications \cite{nise2020control}. The complex plane provides new insights to real problems, allowing mathematicians to impact practical applications. However, most methods for numerically solving differential equations from time $a$ to time $b$ construct a path utilizing only the real line. Here, we highlight particular cases where stepping into the complex plane provides advantages for numerical time integration. 

The idea of taking complex time steps is not new. Early work on time stepping in the complex plane focused on systematically avoiding singularities during numerical integration of singular differential equations \cite{corliss1980integrating}. Complex time steps have also been used in operator splitting schemes. One of the earliest works in this field is \cite{chambers2003symplectic}, where complex substeps were used to avoid negative and unstable substeps, significantly reducing the truncation error. Following this, complex time steps and complex coefficients have been incorporated into operator splitting methods for various parabolic and nonlinear evolution problems \cite{Blanes_Casas_Murua_2024,blanes2022applying,castella2009splitting,auzinger2017practical,casas2021compositions,blanes2010splitting, hansen2009high, casas2022high}. Complex coefficients have been used for higher order methods for linear systems \cite{orendt2009geometry, filatov2006complex} and in  stiff differential algebraic systems \cite{al2006rosenbrock}. More recently, \cite{buvoli2019constructing} used Cauchy's integral formula as inspiration to introduce a general framework for deriving integrators using interpolating polynomials in the complex plane. This results in parallel Backward Differentiation Formula (BDF) methods with improved stability regions. 

The current paper builds on ideas introduced in \cite{george2021walking} and focuses on complex time integrators in the context of stiff problems. These stiff problems often limit the size of stable time steps and necessitate expensive implicit methods or prohibitively small explicit time steps. We show that the additional dimension offered by complex time steps allows us to tune stability regions and design explicit time integrators with significantly enhanced stability properties. In particular, we highlight improved computational efficiency in two key application areas:

\begin{itemize}
    \item \textbf{Complex-valued systems}: Complex time integrators incur minimal additional cost while enabling substantially larger stable time steps.
    \item \textbf{Real-valued stiff systems with spectral gaps and complex eigenvalues}: Complex time steps within the Projective Integration (PI) framework allow efficient integration on slow time scales without resolving fast dynamics, maintaining the computational advantages that PI methods are designed to provide.
\end{itemize}

Our paper is organized as follows: 
\begin{itemize}
    \item \textbf{Section 2}: We illustrate how to use complex time steps to create customized stability regions and recreate the stability regions of the Runge-Kutta-2 and Runge-Kutta-3 methods using complex Forward Euler time steps. 
\item \textbf{Section 3}: We explore asymmetric stability regions, a unique property of complex time integrators that is inaccessible to traditional real-valued integrators. 

We prove that for systems with eigenvalues on the imaginary axis (e.g., the Schrödinger equation), the optimal first-order 2-stage integrator must have complex coefficients (Theorem ~\ref{theorem1}). Furthermore, under mild conditions, we show any real coefficient stability polynomial can be strictly improved by a complex perturbation that enlarges its stability interval along the imaginary axis (Theorem~\ref{theorem2}). Motivated by these results, we numerically compute complex stability polynomials that maximize the stability interval along the imaginary axis for a range of degrees $N$ and orders $p$. 

We demonstrate the advantages of these methods on the nonlinear Schrödinger equation and on time-dependent Kohn--Sham density functional theory simulations \cite{kohn-sham}. In both settings, the enlarged asymmetric stability regions permit significantly larger stable time steps than optimal real-valued integrators, leading to substantial reductions in computational cost while maintaining comparable accuracy. These results highlight the potential of complex time integrators as efficient explicit methods for quantum dynamics simulations, where large stability regions along the imaginary axis are essential  \cite{kohn-sham, kononov2022electron, kang2019pushing, an2020quantum, an2022parallel}.

\item \textbf{Section 4}: We derive the first complex extensions of Projective Integration (PI) schemes, which are ideally suited for efficiently integrating stiff systems with spectral gaps and complex eigenvalues. Complex PI schemes allow stable integration on slow time scales of interest while avoiding the prohibitively small time steps that would be required to resolve fast oscillatory dynamics, thus preserving the computational efficiency that is central to the PI method.

\end{itemize}

The code to reproduce the numerical experiments in this paper is available on GitHub \footnote{\url{https://github.com/Dirivian/complex_time_integrators}} .

\section{Expanded regions of stability with complex time-steps} \label{sec:expanded_stability}

In this section, we illustrate how complex time steps can be leveraged in numerical time integrators to expand the  region of absolute stability. Consider the scalar Dahlquist test problem \cite{hairer1993solving} with parameter $\lambda \in \mathbb{C}$ given in Equation \eqref{eq:dahlquist}.
\begin{align}
    \dot{y} = \lambda y, \quad y(0) = y_0.
    \label{eq:dahlquist}
\end{align}
Equation \eqref{eq:dahlquist} can be integrated with a desired time integrator to obtain the following numerical solution after one time step $\Delta t \in \mathbb{R}_+$.
\begin{align}
    y(t+\Delta t) = \Phi(\lambda \Delta t)y(t),
\end{align}
where $\Phi(z)$ is the stability function of the integrator \cite{leveque2007finite}. The region $z \in \mathbb{C}$ where $|\Phi(z)|\leq 1$ is the region of absolute stability for the integrator.
For the 1-step first-order Forward Euler (FE) method, the 2-step second-order RK2 method and the 3-step third-order RK3 method, the stability regions are given by Equations \eqref{eq:FE}, \eqref{eq:rk2} and \eqref{eq:rk3}, respectively.
\begin{align}
    \Phi_{FE}(z) &= 1 + z, \label{eq:FE} \\
    \Phi_{RK2}(z)    &= 1 + z + \frac{z^2}{2}, \label{eq:rk2} \\
    \Phi_{RK3}(z)    &= 1 + z + \frac{z^2}{2} + \frac{z^3}{3!}. \label{eq:rk3}
\end{align}
The stability regions of the FE, RK2, and RK3 methods are shown in Figure \ref{4_1} (left).

Now, instead of real time steps, we consider integrating Equation \eqref{eq:dahlquist} using a composition of Forward Euler steps with complex time increments. Specifically, we apply up to three successive Forward Euler sub-steps with step sizes $w_1\Delta t$, $w_2\Delta t$, and $w_3\Delta t$, where $w_i \in \mathbb{C}$. The stability polynomial of this 3-step complex Forward Euler (cFE)  method is given by 
\begin{align}
\Phi_{cFE}(z) &= (1 + w_1 z)(1 + w_2 z) (1 + w_3 z)\\&= 1 + (w_1 + w_2 + w_3)z + (w_1 w_2 +w_2w_3 +w_1w_3)z^2 + w_1 w_2 w_3z^3.
    \label{eq:complexFE3}
\end{align}
By picking certain complex values of $w_1, w_2, w_3$ in \eqref{eq:complexFE3}, we can tailor the stability region of the cFE integrator in \eqref{eq:complexFE3} to be the same as that of FE in \eqref{eq:FE}, RK2 in \eqref{eq:rk2}, or RK3 in \eqref{eq:rk3}. 
Each choice of $w_1, w_2, w_3$ represents a different path from $t \in \mathbb{R}$ to $t+ \Delta t \in \mathbb{R}$ through the complex time plane. Because the total step $\Delta t \in \mathbb{R}_+$ is real while the substep sizes $w_i\Delta t$ are complex, each such method is equivalent to an
 explicit Runge-Kutta method with complex Butcher coefficients. Therefore, we use the terms ``complex time stepping'' and ``Runge--Kutta methods with complex coefficients'' interchangeably throughout this paper. We highlight three different choices, sketched also in Figure \ref{4_1} (right):
\begin{enumerate}
    \item $w_1 =  1, (w_2, w_3 = 0)$, leading to the 1-step first-order complex Euler method (cFE1).
    \item $w_1 = \frac{1}{2}+\frac{1}{2}i, w_2 = \frac{1}{2}-\frac{1}{2}i, (w_3 = 0)$, leading to the 2-step second-order complex Euler method (cFE2).
    \item $w_1 \approx 0.186731 +0.480774i, w_2 \approx 0.626538, w_3 \approx 0.186731 - 0.480774i$, leading to the 3-step third-order complex Euler method (cFE3).
\end{enumerate}
The corresponding stability regions in Figure \ref{4_1} (right) recreate the well-known stability regions obtained for the Forward Euler \eqref{eq:FE}, RK2 \eqref{eq:rk2} and RK3 \eqref{eq:rk3} methods. Note that any stability polynomial $\Phi(z)$ with complex coefficients $a_j \in \mathbb{C}$  can be written as 
\begin{align}
    \Phi(z) = 1 +  \sum_{j=1}^{N} a_j z^j =\prod_{i=1}^{N}(1 + w_i z),
\end{align}
 given $\Phi(0) = 1$. Thus, every such polynomial is  realizable as a composition of $N$ complex forward Euler substeps or equivalently, as an explicit Runge-Kutta method with complex coefficients.

\begin{figure}[h!]
\begin{centering}
    \includegraphics[width=5in]{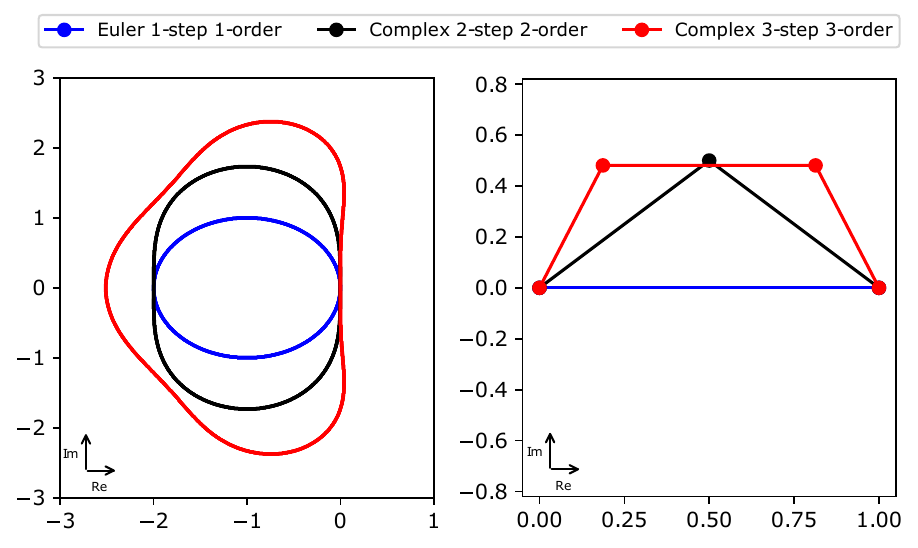}
    \caption{Three different complex time-stepping paths  (right). The 1-step first-order cFE1, 2-step second-order cFE2 and 3-step third-order cFE3 methods share the same stability regions as the FE \eqref{eq:FE}, RK2 \eqref{eq:rk2} and RK3 \eqref{eq:rk3} methods (left). }
     \label{4_1}
\end{centering}
\end{figure}

Since the paths in Figure \ref{4_1} (right) have the same stability polynomials as the corresponding FE, RK2, and RK3 methods, they also have the same order of accuracy for linear autonomous differential equations. Note that the paths sketched in Figure \ref{4_1} (right) are not the only ones. In Figure \ref{fig:pathslinear}, all six 3-step 3rd order paths (for linear autonomous differential equations) are shown along with the two 2-step 2nd order paths and the single step forward Euler path. Previous work \cite{filatov2006complex, orendt2009geometry} has demonstrated that appropriately chosen complex time steps can increase the order of accuracy for linear differential equations. Here, however, our focus is on the stability properties of such complex time integrators.

\begin{figure}[h!]
\begin{centering}
    \includegraphics[width=5.2in]{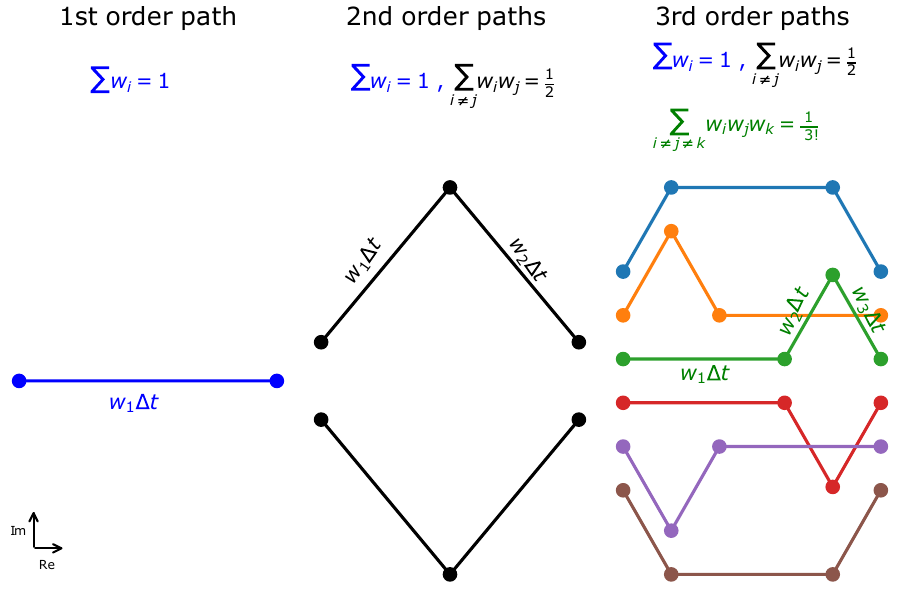}
    \caption{All possible 1-step first-order, 2-step second-order, and 3-step third-order complex Forward Euler methods for linear autonomous differential equations.}
\end{centering}
\label{fig:pathslinear}
\end{figure}

Equation \eqref{eq:complexFE3} suggests the possibility of trading second-order and third-order accuracy in order to obtain customized stability regions. Note, however, that first-order accuracy is necessary to obtain a consistent numerical solution after a total time step of $\Delta t \in \mathbb{R}$. The idea that the order of accuracy can be sacrificed for enhanced stability is not new. Expanded stability regions have been studied in the context of Runge-Kutta methods with real-valued coefficients, where the order of accuracy is sacrificed to obtain a larger stability region \cite{riha1972optimal, lawson1966order,abdulle2000roots, bogatyrev2004effective, ketcheson2013optimal, parsani2013optimized, kubatko2014optimal}. For problems with purely real or purely imaginary eigenvalues, the  optimal stability functions/polynomials have been long-studied \cite{riha1972optimal,bogatyrev2004effective,abdulle2000roots}.

In the past decade, Ketcheson and co-authors \cite{parsani2013optimized, kubatko2014optimal, ketcheson2013optimal} have explored optimal stability functions in the context of systems $\dot{u} = L u$, where $u\in \mathbb{R}^n$ and $L\in \mathbb{R}^{n \times n}$ is a matrix with arbitrary spectrum $\Lambda$. The $N$-stage, order $p$ optimal stability polynomial is defined as the degree-$N$ polynomial $\Phi$ solving the optimization problem:
\begin{align}
    &\text{maximize} \quad h \notag \\
    &\text{subject to} \quad |\Phi(h\lambda)| \leq 1, \quad \forall\, \lambda \in \Lambda,\notag\\
    &\phantom{\text{subject to}} \quad \Phi(z) - e^z = O(z^{p+1}), \quad z \to 0.  \label{eqn:opt_stab_prob} 
\end{align}
In \cite{ketcheson2013optimal}, Ketcheson and Ahmadia present an algorithm implemented in RKOpt \cite{ketcheson2020rk} and publicly available through NodePy \cite{ketcheson2020nodepy} to solve this problem and obtain the optimal stability polynomial admitting the largest time step for a given spectrum $\Lambda$. However, the extension of this technique to complex time steps and complex coefficients is currently missing.

In Figure \ref{4_2}, we illustrate a 3-step second-order complex Forward Euler method (cFE3-2) (blue) and 3-step first-order complex Forward Euler method (cFE3-1) (yellow) with expanded stability along the negative real axis compared to the 3-step third-order complex Forward Euler method (cFE3) (red). Note that the cFE3-1 method is achieved through purely real time steps.

\begin{figure}[h!]
\begin{centering}
    \includegraphics[width=5in]{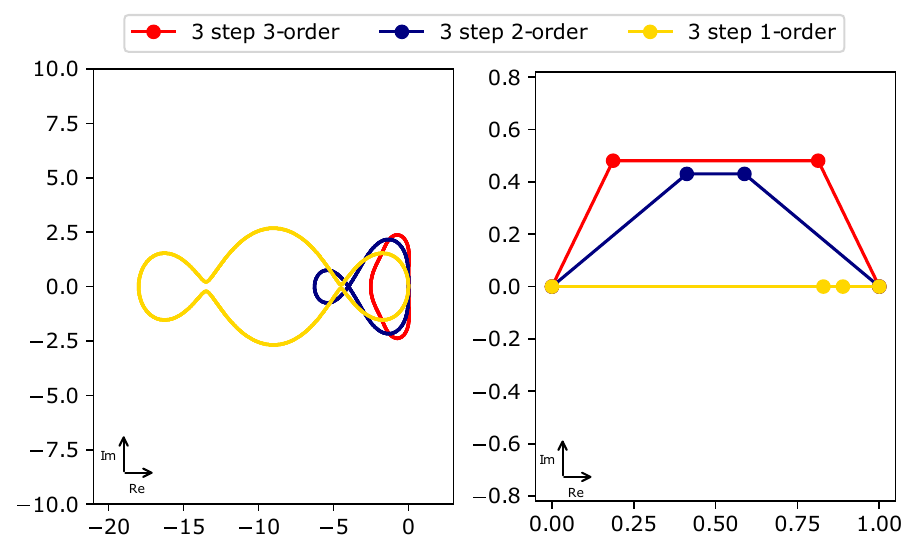}
    \caption{By sacrificing accuracy, the stability region can be expanded (left) using complex time steps, as demonstrated by these optimized 3-step complex Euler paths (right) with third-order (cFE3), second-order (cFE3-2), and first-order (cFE3-1) accuracy.}
\label{4_2}
\end{centering}
\end{figure}

To our knowledge, this work is the first to use complex time steps to tailor stability polynomials. At first glance, the advantages of complex time integration over traditional real-valued Runge-Kutta methods may not be obvious, particularly given the additional cost of complex arithmetic. Moreover, not all real-valued nonlinear differential equations admit natural extensions into the complex plane. Nevertheless, our approach provides clear benefits in several contexts commonly encountered in applications, such as systems with asymmetric eigenvalue spectra, complex-valued quantum dynamics, and stiff problems with spectral gaps. 

The following sections describe how tailored complex time integration enables larger time steps, expands stability regions, and improves computational efficiency. The next section highlights the benefit of asymmetric stability regions which are inaccessible to traditional real-valued integrators but achievable with complex time steps.

\section{Asymmetrically expanded regions of stability} \label{sec:assymetric_stability}

In this section, we explore a unique feature of complex time integration: asymmetric stability regions. Complex time steps can result in stability polynomials with complex valued coefficients. The associated stability regions can then be asymmetric about the real axis, a feature that is not present for real-valued coefficients.  For real-valued problems, this asymmetry can be disadvantageous since the eigenvalues of real-valued problems appear as complex conjugate pairs symmetric about the real axis. In contrast, for complex-valued problems which have asymmetric eigenvalue spectra, complex coefficients can yield substantially larger stability regions in the relevant direction. A prototypical application is the complex-valued time-dependent Kohn--Sham equation, which describes ultra-fast electron dynamics, for which there is an active search for explicit integrators with expanded stability regions \cite{kononov2022electron,an2020quantum,an2022parallel,kang2019pushing}.

We first illustrate asymmetric stability regions on a generic linear system. We then prove that, for systems whose eigenvalues lie on either the positive or negative imaginary axis, the optimal explicit integrators necessarily have complex coefficients (Theorems~\ref{theorem1} and ~\ref{theorem2}). Finally, we demonstrate the resulting computational advantages on two complex-valued applications: the nonlinear Schrödinger equation and the three-dimensional time-dependent Kohn--Sham equation.

\subsection{A simple example}

Consider the linear system $\dot{u} = A u$ for  $A \in \mathbb{C}^{n\times n}$ with eigenvalues visualized by the red dots in Figure \ref{4_3equation} (left). The naive 3-step 3rd order complex integrator cFE3 from the previous section (resulting in the stability domain bounded by the red curve) is obviously not optimally designed and would force a small time step size.
\begin{figure}
\begin{centering}
    \includegraphics[width=5in]{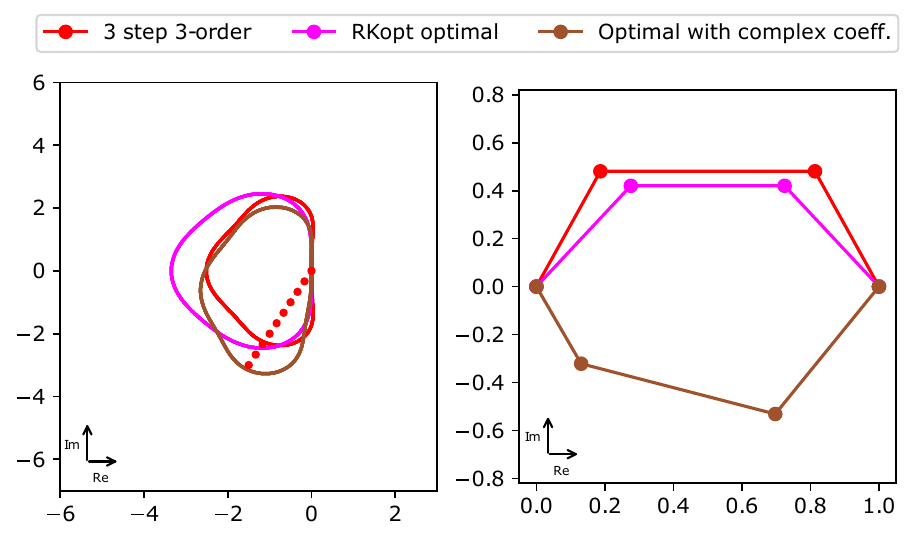}
    \caption{Complex coefficients in optimal stability polynomials often result in asymmetrical stability regions. The 3-step third-order (cFE3) complex integrator (red curves) is not optimal for the eigenvalues (red dots, left panel). RKOpt  \cite{ketcheson2020rk} yields the stability polynomial in Eq. \eqref{eq:RKOPT_stab}, corresponding to the optimal 3-stage second-order real integrator (magenta curves). An optimized 3-stage complex integrator (brown curves) achieves the stability polynomial  in Eq. \eqref{eq:complexOPT_stab} which allows the largest possible time steps for these eigenvalues.}
\label{4_3equation}
\end{centering}
\end{figure}

The magenta curve in Figure \ref{4_3equation} (left) sketches the improved but still symmetric stability region of the stability polynomial given by
\begin{align}
    \Phi(z) = 1 + z + \frac{z^2}{2} + 0.1134 z^3
    \label{eq:RKOPT_stab},
\end{align}
which is the optimal stability polynomial with real coefficients for a 3-stage second-order method obtained using RKOpt \cite{ketcheson2020rk} for these eigenvalues. 

As discussed in the previous section, we can recreate this stability region using three complex Forward Euler steps, leading to the magenta path in Figure \ref{4_3equation} (right). We can also choose three complex time steps such that the stability polynomial has complex-valued coefficients. This allows us to numerically find the better stability polynomials. The particular choice of $w_1 \approx  0.1306-0.3217i, w_2 \approx  0.5667-0.2097i,  w_3 \approx 0.3027+0.5314i$ leads to the stability polynomial in Equation \eqref{eq:complexOPT_stab}, visualized in Figure \ref{4_3equation} by the brown curve.
\begin{align}
    \Phi(z) = 1 + z + \frac{z^2}{2} + (0.1134- 0.06i) z^3.
\label{eq:complexOPT_stab}
\end{align}
The brown curve clearly demonstrates the advantage of complex time integration. The 3-step 2nd order complex integrator produces an asymmetric stability region that fully encompasses the relevant eigenvalues of the system (red dots), allowing larger stable time steps for the same spectrum compared to the 3-stage 2nd order integrator with real coefficients. This illustrates that complex coefficients can be used to expand the stability region in the directions where it is most needed, rather than being constrained by symmetry requirements as in the real-valued case.

We note that the 3-stage 2nd order optimized real integrator remains the optimal choice if one has to take time steps along both these eigenvalues and their complex conjugates, as occurs in real-valued systems with complex conjugated pairs of eigenvalues. Overall, these results illustrate that complex time integration combined with optimized stability polynomials provides a clear benefit for problems with asymmetric complex eigenvalue spectra, enabling larger stable time steps and more efficient simulations.

\subsection{Optimally stable integrators for the Schrödinger equation}

To demonstrate the expanded stability benefits of complex time steps on a practical system, we consider a classical complex-valued problem: the linear Schrödinger equation \cite{sakurai2020modern} given by Equation \eqref{eq:linear_Schrodinger},
\begin{align}
  i u_t +\frac{1}{2}u_{xx} = 0.
    \label{eq:linear_Schrodinger}
\end{align}

For typical space discretizations, the linear Schrödinger equation has eigenvalues $\lambda_i \in \mathbb{C}$ along the negative imaginary axis (see Figure \ref{4_4equation}). We show that for systems with all eigenvalues on either the positive or negative imaginary axis, the optimal first-order 2-stage integrator i.e., the  first-order 2-stage integrator that allows the maximal stable time step size,  must have complex coefficients.

\begin{theorem}\label{theorem1}
Consider the linear system $\dot{u} = Au$, where $u \in \mathbb{C}^n$ and $A \in \mathbb{C}^{n \times n}$ has eigenvalues lying entirely on the positive imaginary axis ($\lambda = +i|\lambda|$) or entirely on the negative imaginary axis ($\lambda = -i|\lambda|$). For this case, the optimization problem Equation \eqref{eqn:opt_stab_prob} can be written as
\begin{align}
    &\text{maximize} \quad Y_{\max} \notag \\
    &\text{subject to} \quad |\Phi(\pm iy)| \leq 1, \quad \forall\, y \in [0, Y_{\max}], \quad Y_{\max} \in  \mathbb{R},  \notag\\
    &\phantom{\text{subject to}} \quad \Phi(z) - e^z = O(z^{p+1}), \quad z \to 0 
    \label{eqn:opt_stab_problem_3_1}
\end{align}
where the sign in $|\Phi(\pm iy)| \leq 1$ denotes either eigenvalues on the positive or negative imaginary axis, respectively.

Among all stability polynomials of a general 2-stage, first-order method, the optimal stability polynomial, i.e., the polynomial of degree at most $N = 2$ maximizing $Y_{\max}$ in Equation \eqref{eqn:opt_stab_problem_3_1} with $p = 1$ has complex-valued coefficients and is given by

\[
  \Phi_{c_{\pm}}(z) = 1 + z + \left(\frac{1}{2} \pm \frac{1}{2}i\right)z^2,
  \quad Y_{\max} = 2,
\]
where $\pm$ corresponds to eigenvalues either entirely on the positive or entirely on the negative imaginary axis, with stability intervals $[0, 2i]$ and  $[-2i, 0]$ respectively. 

When restricted to polynomials with real-valued coefficients, the optimal stability polynomial for eigenvalues on either axis is
\[
  \Phi_r(z) = 1 + z + z^2, \quad Y_{\max} = 1,
\]
with the symmetric stability interval $[-i, i]$.

Hence the optimal complex polynomial strictly enlarges the stability interval by a factor of two compared to the optimal real polynomial \cite{ketcheson2013optimal}, doubling the maximum stable step size.
\end{theorem}

\begin{proof}
Without loss of generality, consider a system $\dot{u} = A u$ with all the eigenvalues $\lambda$ of $A$ on the positive imaginary axis, i.e., $\lambda = +i|\lambda|$. 
The stability polynomial for a general 2-stage  first-order integrator is given by
\begin{align}
\label{eq:schro_stab}
    \Phi(z) = 1 + z + k z^2,
\end{align}
where the value of the coefficient $k$ needs to be optimized so that $|\Phi(z)| \leq 1$ is satisfied on the maximal interval along the positive imaginary axis, enabling the maximal stable time step. Note that consistency requires first-order accuracy \cite{hairer1993solving}, corresponding to $p = 1$ in \eqref{eqn:opt_stab_problem_3_1}, which fixes the constant and the coefficient of the linear term in $\Phi(z)$ from \eqref{eq:schro_stab} to $1$, leaving only $k$ as a free parameter.

For the stability polynomial with purely real coefficients, i.e., $k \in \mathbb{R}$, the stability condition $|\Phi(z)| \leq 1$    on $ z = iy \in [0, iY_{\max}]$ becomes 
\begin{align}
\label{eq:phi2}
    \left(k^2y^2+(1-2k)\right) y^2 \leq 0,
\end{align}
For the integrator to be stable, we thus require $y^2 \in [0, \frac{2k-1}{k^2}]$. The maximal value of $y$ is $Y_{\max} = 1$, which is achieved for $k = 1$. Hence, for real-valued integrators, the optimal 2-stage, first-order real coefficient stability polynomial for this system is 
\begin{align}
    \Phi_r(z) = 1 +z + z^2
    \label{eq:stabschro_real}
\end{align} and the maximum stable step is $\Delta t = \frac{Y_{\max}}{|\lambda_{\max}|}= \frac{1}{|\lambda_{\max}|}$, where $\lambda_{\max}$ is the largest-magnitude eigenvalue. Since Equation \eqref{eq:phi2} has only even powers of $y$, we recover the same optimal 2-stage, first-order real coefficient stability polynomial and maximal stable step for the case of eigenvalues on the negative imaginary axis.

For the stability polynomial with the complex coefficient $k = k_r + i k_i  \in \mathbb{C}$, $k_r, k_i \in \mathbb{R}$, the stability condition translates to 
\begin{align}
\label{eq:phi1b}
    (k_r^2+k_i^2)y^2- 2 k_iy+ (1-2k_r) =: f(y,k_r,k_i) \leq 0.
\end{align}
The stability condition must hold at $y=0$. Thus, $f(0,k_r,k_i) \leq 0$ which imposes the constraint  $k_r \geq 1/2$.

For a given $y$, the unconstrained minimum of $f(y,k_r,k_i)$ occurs at $\partial f(y,k_r,k_i)/ \partial k_r = 0$ or $k_r=1/y^2$. For $y>\sqrt{2}$ , this lies outside the
feasible region. Since $f(y,k_r,k_i)$ is a convex quadratic in $k_r$, the constrained minimum is then attained at the boundary point,
$k_r=1/2$.

After fixing $k_r=\frac12$, the upper root of $f(y, k_r=\frac12,k_i)$ is given by $y = 2 k_i/(\frac{1}{4}+k_i^2)$, which attains a maximal value of $2$ for $k_i = \frac{1}{2}$. 

 
Thus, the optimal coefficient $k =\frac{1}{2}+\frac{1}{2}i$ achieves a stability domain of $y \in [0, 2]$. Hence, the optimal 2-stage, first-order complex coefficient stability polynomial for this system is 
\begin{align}
    \Phi_{c_{+}}(z) = 1 +z + \left(\frac{1}{2}+\frac{1}{2}i\right)z^2
\end{align}
 and the maximum stable step is $\frac{2}{|\lambda_{\max}|}$ which is twice as large as its real equivalent.

Similarly, if all the eigenvalues lie on the negative imaginary axis like in the case of the linear Schrödinger equation, the optimal 2-stage, first-order complex coefficient stability polynomial is 
\begin{align}
    \Phi_{c_-}(z) = 1 +z + \left(\frac{1}{2}-\frac{1}{2}i\right)z^2
    \label{eq:stabschro_complex}
\end{align}
and the maximum stable step is also $\frac{2}{|\lambda_{\max}|}$.
\end{proof}

The stability region using complex coefficients thus covers double the length on the negative imaginary axis compared to the one obtained using purely real coefficients, enabling stable time steps twice as large as also shown in Figure \ref{4_4equation}.
\begin{figure}
\begin{centering}
    \includegraphics[width=6in]{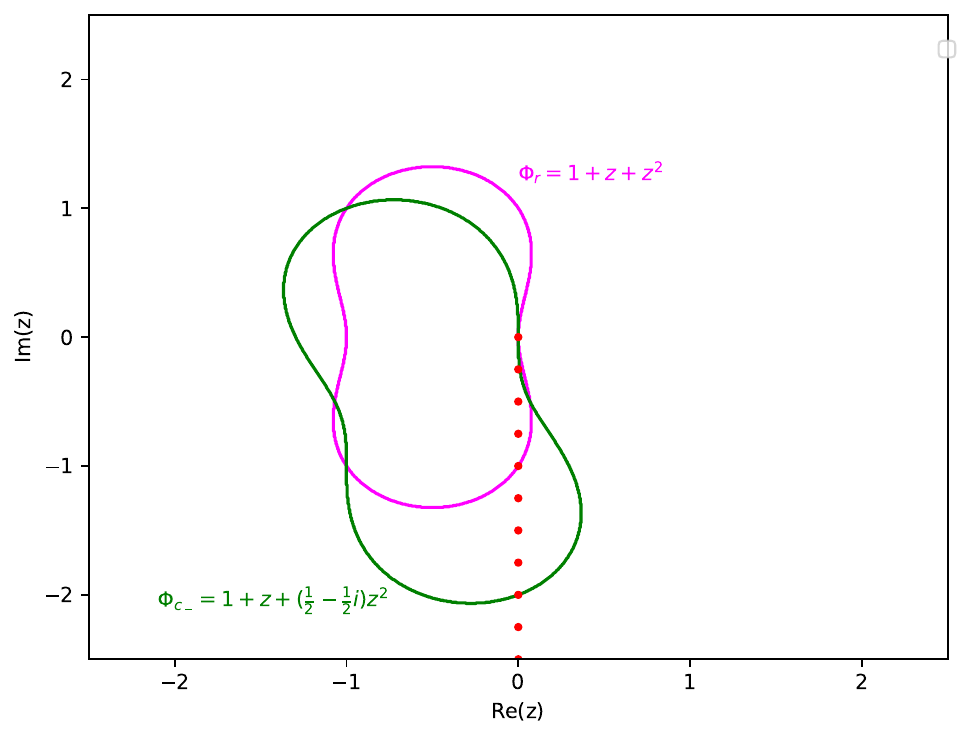}
    \caption{Eigenvalues of the linear Schrödinger equation lie along the negative imaginary axis (red dots). The optimal two-stage, two-step, first-order stability polynomials with real coefficients $\Phi_r$(Eq.~\eqref{eq:stabschro_real}, magenta) and with complex coefficients  $\Phi_{c_-}$ (Eq.~\eqref{eq:stabschro_complex}, green) yield distinct stability regions. Using complex coefficients with $\Phi_{c_-}$ expands the stability region, enabling stable time steps up to twice as large as those for the real coefficient integrator $\Phi_r$.}
\label{4_4equation}
\end{centering}
\end{figure} 

The following theorem generalizes this idea further. Stability polynomials with real coefficients, for systems with eigenvalues either on the positive or negative imaginary axis, can be improved by introducing a small complex perturbation, thereby enlarging their stability regions along the imaginary axis.

\begin{theorem}\label{theorem2}

Let 

\[
  \Phi_r(z) = 1 + \sum_{j=1}^{p} \frac{z^j}{j!} + \sum_{j= p+1}^{N} a_j z^j,
  \quad a_j \in \mathbb{R},
\]
be the $N$-stage, $p$th-order optimal real stability polynomial solution to Equation \eqref{eqn:opt_stab_problem_3_1} with maximal 
stability interval length $Y_R$ along the positive imaginary axis among all $N$-stage, $p$th-order real 
stability polynomials. Consistency requires first-order accuracy \cite{hairer1993solving} and thus, $p \geq 1$.

Assume near $y = 0$,
\[
  |\Phi_r(iy)|^2 - 1 = cy^M + O(y^{M+1}), \quad c < 0,
\]
and the number of points $y \in (0, Y_R]$ at which $|\Phi_r(iy)| = 1$ is at
most $N - M + 1$, then  there exists a complex polynomial $q(z)=\sum_{t=M}^{N}w_tz^t$ with $w_t\in\mathbb C$ such that for all
sufficiently small $\varepsilon>0$, the perturbed polynomial
\[
  \Phi_\varepsilon(z) = \Phi_r(z) + \varepsilon\, q(z)
\]
satisfies
\[
  |\Phi_\varepsilon(iy)| \leq 1 \quad \forall\, y \in [0, Y_C],
\]
for some $Y_C > Y_R$  and is at least $p$th order.  Hence, introducing complex coefficients strictly enlarges the attainable stability region along the imaginary axis without sacrificing the order of accuracy.

\end{theorem}

\begin{proof}
Let
\begin{align}
     \rho_{r}(y) =  |\Phi_r(iy)|^2.
\end{align}

The stability condition becomes
\begin{align}
    \rho_{r}(y) \leq 1, \quad \forall\, y \in [0, Y_R].
\end{align}

 Since $\Phi_r$ has real coefficients, we  have
\begin{align}
      \rho_r(y) = \Phi_r(iy)\Phi_r(-iy) = \bigl(e^{iy}+O(y^{p+1})\bigr)\bigl(e^{-iy}+O(y^{p+1})\bigr) = 1 + O(y^{p+1}).
      \label{eq:pplus1}
\end{align}
Near $y = 0$, we have
 \begin{align}
     \rho_r(y)=1+cy^M+O(y^{M+1}), \quad c < 0.
\label{eq:M} \end{align}

From Equations \eqref{eq:pplus1} and \eqref{eq:M}, we get $M \geq p+1$. Consider the perturbed polynomial $\Phi_\varepsilon(z) = \Phi_r(z) + \varepsilon q(z)$
for  $\varepsilon > 0$. 
Since $q(z)=\sum_{t=M}^{N}w_tz^t$ and $M \geq p+1$, $\Phi_\varepsilon(z)$ has order at least $p$ and the perturbation cannot affect the lower order terms. Furthermore, we can write
\begin{align*}
  \rho_\varepsilon(y) &= |\Phi_\varepsilon(iy)|^2
  = \bigl(\Phi_r(iy) + \varepsilon q(iy)\bigr)\overline{\bigl(\Phi_r(iy) + \varepsilon q(iy)\bigr)} \\
  &= \rho_{r}(y) + 2\varepsilon \operatorname{Re}\bigl(\Phi_r(iy)\,\overline{q(iy)}\bigr) + \varepsilon^2|q(iy)|^2.
\end{align*}

\textbf{Behavior of  $\rho_\varepsilon$ near 0: }
From Equation \eqref{eq:M}, we see that there exists a neighborhood $[0, \delta)$ where $0<\delta < 1$ such that

 \[
\rho_r(y)- 1 \leq  \frac{c}{2}y^M, \quad c < 0.
\]

We examine the effect of the complex perturbation $q(iy)$ in the neighborhood $[0, \delta)$.  We have

\begin{align*}
     \rho_\varepsilon(y) - 1 &=   \rho_r(y) -1 + 2\varepsilon\operatorname{Re}\bigl(\Phi_r(iy)\,
  \overline{q(iy)}\bigr) + \varepsilon^2|q(iy)|^2 \\
  & \leq  \frac{c}{2}y^M + 2\varepsilon\operatorname{Re}\bigl(\Phi_r(iy)\,
  \overline{q(iy)}\bigr) + \varepsilon^2|q(iy)|^2.
\end{align*}

If the lowest-degree non-zero term of $q(z)$ is $w_s z^s$ where $s \geq M$, then $\operatorname{Re}(\Phi_r(iy)\,\overline{q(iy)}) = \operatorname{Re}(\overline{w_s}(-iy)^s)+ O(y^{s+1})$, so the perturbation contributes at order $y^s$. Therefore, for $y\in[0,\delta)$,
\begin{align}
  \rho_\varepsilon(y)-1 \le
  \frac{c}{2}y^M + 2\varepsilon\operatorname{Re}\bigl(\overline{w_s}(-iy)^s\bigr)
  + \varepsilon\,O(y^{s+1}) + \varepsilon^2 O(y^{2s}).
  \label{eq:bound_near_0}
\end{align}

Then for $y\in[0,\delta)$, since $\delta \leq 1$, we have the following bounds:

\[|\operatorname{Re}(\overline{w_s}(-iy)^s)|\le |w_s|y^s\le |w_s|y^M, \quad y^{s+1}\le y^M, \quad y^{2s}\le y^M. \]

The last three terms in Equation \eqref{eq:bound_near_0} are thus bounded by $y^M\cdot O(\varepsilon)$, with the $O(\cdot)$ constants
depending on $\Phi_r$ and $q$ but not on $y$ or $\varepsilon$. 
 Hence, 
\[
  \rho_\varepsilon(y)-1 \le y^M\Bigl(\frac{c}{2}+O(\varepsilon)\Bigr)\le 0
  \quad\text{for all } y\in[0,\delta)
\] 
for all  $0 < \varepsilon < \varepsilon_0$ for some sufficiently small $\varepsilon_0$.

\textbf{Behavior of  $\rho_\varepsilon$ near $Y_R$ and other points $y_j$ where $|\Phi_r(iy_j)| = 1$: }
Since $\Phi_r$ is optimal, $\rho_{r}(Y_R) = 1$, so $\Phi_r(iY_R) = e^{i\theta}$ for some
$\theta \in \mathbb{R}$. Evaluating $\rho_\varepsilon$ at $y = Y_R$,
\[
  \rho_\varepsilon(Y_R) = 1 + 2\varepsilon \operatorname{Re}\bigl(e^{i\theta}\,\overline{q(iY_R)}\bigr) + O(\varepsilon^2).
\]
For $\rho_\varepsilon(Y_R) < 1$, it suffices that $\operatorname{Re}\bigl(e^{i\theta}\,\overline{q(iY_R)}\bigr) <0$. Choosing $q(iY_R) = -e^{i\theta}$ gives
\[
  \rho_\varepsilon(Y_R) = 1 - 2\varepsilon + \varepsilon^2 < 1
\]
for sufficiently small $\varepsilon < 2$. The same argument applies at any interior
point $y_j \in (0, Y_R)$ where $\rho_{r}(y_j) = 1$. Writing $\Phi_r(iy_j) = e^{i\theta_j}$
and setting $q(iy_j) = -e^{i\theta_j}$ ensures $\rho_\varepsilon(y_j) < 1$.

Let us define $f(y) = \operatorname{Re}\bigl(\Phi_r(iy)\,\overline{q(iy)}\bigr)$. At any interior
point $y_j \in (0, Y_R)$ where $\rho_{r}(y_j) = 1$, we have $f(y_j) = -1$. Since $f$ is continuous because $\Phi_r$ and $q$ are polynomials, there exists $\delta_j >0$ such that $(y_j-\delta_j,y_j+\delta_j)\subset(0,Y_R)$ and 
\[ f(y) \leq -1/2, \quad |y-y_j| < \delta_j.\]
Since $f,\Phi_r, q, y_j$ do not depend on $\varepsilon$, $\delta_j$ and the neighborhood $(y_j-\delta_j,y_j+\delta_j)\subset(0,Y_R)$  are independent of $\varepsilon$. In this neighborhood, we have
 \[ \rho_\varepsilon(y) -1 = \rho_{r}(y)-1 + 2\varepsilon f(y) + \varepsilon^2|q(iy)|^2 \leq -\varepsilon + \varepsilon^2 Q_j,\] 
 where $Q_j = \max_{|y-y_j| \leq \delta_j} |q(iy)|^2$. Thus, $\rho_\varepsilon(y) \leq 1$  in $(y_j-\delta_j, y_j+\delta_j)$ for all $0 < \varepsilon < 1/Q_j$.  
 
 Similarly, we can show there is a neighborhood $(Y_R-\delta_R, Y_R]$ independent of $\varepsilon$ where $\rho_\varepsilon(y) \leq 1$ for $0 <\varepsilon < 1/Q_R$ where $Q_R = \max_{y \in [Y_R-\delta_R, Y_R]} |q(iy)|^2$.

\textbf{Limit on the number of points where $|\Phi_r(iy_j)| = 1$: } The polynomial $q(z) = \sum_{t=M}^{N} w_t z^t$
 has $N-M+1$ free coefficients. Since the nodes $iy_j$ are distinct and non-zero and the exponents $t$ are consecutive,  the coefficient matrix
$\bigl[(iy_j)^t\bigr]$ is a generalized Vandermonde matrix of full row rank. So, the interpolation problem $q(iy_j) = -e^{i\theta_j}$  is solvable whenever the number of points $y_j \in (0, Y_R]$ at which $|\Phi_r(iy_j)| = 1$  does not exceed the number $N-M+1$ of free coefficients, which holds by assumption.

\textbf{Conclusion:} We have shown $\rho_\varepsilon(y) \leq 1$ on $[0, \delta)$ for some $\delta > 0 $ independent of $\varepsilon$. Similarly, we have shown $\rho_\varepsilon(y) \leq 1$ on the  neighborhoods $(Y_R - \delta_R, Y_R]$ and  $(y_j - \delta_j, y_j + \delta_j)$  for some $\delta_j , \delta_R > 0$ independent of $\varepsilon$. 

Let $\mathcal{U}$ be the union of $[0, \delta)$, $(Y_R - \delta_R, Y_R]$ and all neighborhoods $(y_j - \delta_j, y_j + \delta_j)$. We have shown $\rho_\varepsilon(y) \leq 1$ in $\mathcal{U}$. The complement $[0, Y_R] \setminus \mathcal{U}$ is closed and bounded, thus compact. Since we have $\rho_{r}(y) < 1$ in the compact set $[0, Y_R] \setminus \mathcal{U}$,  there exists some positive $\eta$ such that 
\[
  \inf_{[0, Y_R] \setminus \mathcal{U}} (1 - \rho_{r}(y)) = \eta > 0.
\]
i.e.,
\[
   \rho_{r}(y) + \eta \leq 1, \quad y \in [0, Y_R] \setminus \mathcal{U}.
\]

Recall $\rho_\varepsilon - \rho_{r}$ is a polynomial in $y$ whose coefficients vanish as $\varepsilon \to 0$, and thus
$\rho_\varepsilon \to \rho_{r}$ uniformly on the compact set $[0, Y_R] \setminus \mathcal{U}$.
Hence there exists $\varepsilon_U>0$ such that
\[
  |\rho_\varepsilon(y)-\rho_r(y)|<\eta \quad \text{for all } y\in[0,Y_R]\setminus\mathcal U
  \text{ and all } 0<\varepsilon<\varepsilon_U .
\]
Since $q$, $\delta$, $\delta_j$, $\delta_R$, $\mathcal U$ and $\eta$ do not depend on $\varepsilon$, we can fix $\varepsilon$ with the bounds found earlier in this proof,
\[
  0<\varepsilon<\min\Bigl\{2, \varepsilon_0,\ \tfrac{1}{Q_1},\dots,\tfrac{1}{Q_{j}},\dots, \tfrac{1}{Q_R},\ \varepsilon_U\Bigr\}.
\]
 For this $\varepsilon$, we have
$\rho_\varepsilon(y) \le 1$ on $\mathcal U$ by the arguments above. Furthermore, on $[0,Y_R]\setminus\mathcal U$, we have
\[
  \rho_\varepsilon(y) < \rho_r(y)+\eta \leq 1 .
\]

We can thus conclude $\rho_\varepsilon(y) \leq 1$ in $[0, Y_R]$. Since we have $\rho_\varepsilon(Y_R) < 1$, by continuity of $\rho_\varepsilon(y)$ in $y$, there exists some $Y_C>Y_R$ such that  $|\Phi_\varepsilon(i y)|\leq 1$ for all $y\in[0,Y_C]$.  
\end{proof}

\remark{The corresponding result for eigenvalues lying entirely on the negative imaginary axis ($\lambda = -i|\lambda|$) follows by symmetry. 

Since any real-valued stability polynomial satisfies
$\Phi(-iy)=\overline{\Phi(iy)}$, its stability interval along the
imaginary axis is symmetric. The optimization problem for
eigenvalues on the negative imaginary axis is thus equivalent to the one on the positive imaginary axis, and $\Phi_r$ remains the optimal
real-valued stability polynomial solution with the same stability interval length $Y_R$.  Replacing the
perturbation polynomial $q$ by its coefficient-wise conjugate
$\bar q$ gives \[
|\Phi_r(-iy)+\varepsilon\bar q(-iy)|
=
|\overline{\Phi_\varepsilon(iy)}|
=
|\Phi_\varepsilon(iy)|
\leq 1,
\]
for $y\in[0,Y_C]$.  Hence, the same enlargement is achieved along the
negative imaginary axis using $\bar q$, yielding the stability interval
$[-iY_C,0]$ where $Y_C > Y_R$.}

\begin{remark}
\label{remark:theorem2}
The condition in Theorem~\ref{theorem2} that the number of points
$y_j\in(0,Y_R]$ satisfying $|\Phi_r(iy_j)|=1$ is at most $N-M+1$ is
sufficient but not necessary. In our proof, we make two conservative choices. First, at each $y_j$, we impose the equality $q(iy_j)=-e^{i\theta_j}$, although stability requires only
 the inequality $\operatorname{Re}\bigl(e^{i\theta_j}\,\overline{q(iy_j)}\bigr)<0$. Second, we restrict the perturbation to degrees $s\ge M$, so that the leading term $c\,y^{M}$ ($c<0$) dominates as $y \to 0$ irrespective of the sign of the perturbation. A lower-degree term ($p+1\le s<M$) with favourable sign can also preserve stability near the origin. Thus, a method with more than  $N - M + 1$ points $y_j$ may still admit an enlarging complex perturbation.
\end{remark}

Theorem~\ref{theorem1} provides an explicit result for two-stage, first-order methods applied to systems whose eigenvalues lie on the positive or negative imaginary axis. It shows that optimizing over complex coefficients doubles the attainable stability interval relative to the optimal real coefficient stability polynomial. Theorem~\ref{theorem2} generalizes this further, proving that, under mild assumptions, an appropriate complex perturbation of an optimal real stability polynomial yields a strictly larger stability interval along the imaginary axis. Together, these results suggest that complex-coefficient methods can attain larger stability regions than their real coefficient counterparts particularly for problems with asymmetric spectra. Consequently, optimized complex-coefficient integrators may permit larger time steps and more efficient time integration.

Motivated by these theoretical results, we numerically revisit the optimization problem described in Theorem \ref{theorem1} over stability polynomials for order $p = 1, 2, 3, 4 $ and degree $N = 2, 3, 4, 5$.
\begin{align}
    &\text{maximize} \quad Y_{\max} \notag \\
    &\text{subject to} \quad |\Phi(iy)| \leq 1, \quad \forall\, y \in [0, Y_{\max}], \quad Y_{\max} \in  \mathbb{R}  \notag\\
    &\phantom{\text{subject to}} \quad \Phi(z) - e^z = O(z^{p+1}), \quad z \to 0. 
    \label{eqn:opt_stab_problem_3_3}
\end{align}

For real coefficients, this optimization problem was solved by Ketcheson and Ahmadia \cite{ketcheson2013optimal} and the optimal stability intervals  are reported in Table 2 of their paper. Table~\ref{tab:stability_intervals} compares the optimal stability intervals $Y_C$ and $Y_R$ for complex and real polynomials respectively, for $p = 1, 2, 3, 4$ and $N = 2, 3, 4, 5$ with the  corresponding optimal coefficients  given in Table~\ref{tab:coefficients}. By construction, these coefficients optimize the stability interval over
the positive imaginary axis, $|\Phi(iy)|\le1$ for all $y\in
[0,Y_{\max}]$. The same stability intervals hold over the
negative imaginary axis, $|\Phi(iy)|\le1$ for all $y\in [-Y_{\max},0]$, using the conjugated polynomial
$\Phi(z)=1+z+\sum_{k=2}^{N}\bar a_k z^k$, i.e.,\ by using the complex conjugate of every coefficient in Table~\ref{tab:coefficients}.

\begin{table}[h]
\centering
\caption{Stability intervals for complex and real (in parentheses) 
optimal stability polynomials that are the solutions to Equation~\eqref{eqn:opt_stab_problem_3_3} for degree $N$ and order $p$. Entries show $Y_C\ (Y_R)$, 
where $Y_C$ and $Y_R$ denote the length of the stability interval along the 
imaginary axis for the optimal complex and real stability polynomials, 
respectively. Dashes indicate cases where $p \geq N$.}
\label{tab:stability_intervals}
\begin{tabular}{ccccc}
\toprule
& \multicolumn{4}{c}{Order $p$} \\
\cmidrule(lr){2-5}
Degree $N$ & $p=1$ & $p=2$ & $p=3$ & $p=4$ \\
\midrule
2 & $\complexval{2}\ (1)$ & --- & --- & ---\\
3 & $\complexval{5.19}\ (2)$ & $\complexval{2}\ (2)$ & --- & ---\\
4 & $\complexval{9.6}\ (3)$ & $\complexval{4}\ (2.83)$ & $\complexval{3.1}\ (2.83)$ & ---\\
5 & $\complexval{15.3}\ (4)$ & $\complexval{6.5}\ (4)$ & $\complexval{4.4}\ (3.9)$ & $\complexval{4.39}\ (3.46)$\\
\bottomrule
\end{tabular}

\end{table}

\begin{table}[t]
\centering
\small
\renewcommand{\arraystretch}{1.3}
\setlength{\tabcolsep}{6pt}

\caption{Coefficients of the optimal complex stability polynomials 
$\Phi(z)=1+z+\sum_{k=2}^{N}a_kz^k$
that are the solutions to Equation \eqref{eqn:opt_stab_problem_3_3} for degree $N$ and order $p$. The coefficients $a_0=a_1=1$ are omitted. Dashes indicate coefficients that do not exist.}
\label{tab:coefficients}

\begin{tabular}{ccccc}
\toprule
$(N,p)$ & $a_2$ & $a_3$ & $a_4$ & $a_5$ \\
\midrule

$(2,1)$ &
$\tfrac12+\tfrac12 i$ &
--- & --- & --- \\

$(3,1)$ &
$\tfrac12+0.4945i$ &
$-0.0580+0.0963i$ &
--- & --- \\

$(3,2)$ &
$\tfrac12$ &
$\tfrac14$ &
--- & --- \\

$(4,1)$ &
$\tfrac12+0.4930i$ &
$-0.0744+0.1255i$ &
$-0.00788-0.00354i$ &
--- \\

$(4,2)$ &
$\tfrac12$ &
$0.2800$ &
$0.0313+0.0540i$ &
--- \\

$(4,3)$ &
$\tfrac12$ &
$\tfrac16$ &
$0.0355+0.0110i$ &
--- \\

$(5,1)$ &
$\tfrac12+0.4945i$ &
$-0.0820+0.1379i$ &
$-0.01242-0.00560i$ &
$0.0001359-0.0003622i$ \\

$(5,2)$ &
$\tfrac12$ &
$0.3024-0.0001i$ &
$0.04839+0.0830i$ &
$-0.00618+0.005792i$ \\

$(5,3)$ &
$\tfrac12$ &
$\tfrac16$ &
$\tfrac1{24}+0.0015i$ &
$0.0056+0.00359i$ \\

$(5,4)$ &
$\tfrac12$ &
$\tfrac16$ &
$\tfrac1{24}$ &
$0.00603+0.00359i$ \\

\bottomrule
\end{tabular}
\end{table}

In all cases except $(p,N)=(2,3)$, the optimal complex polynomial achieves a strictly larger stability interval than the optimal real polynomial. For the case where $p = 2$, $N = 3$, we can derive the optimal real polynomial 
\begin{align}
    \Phi_r(iy) &= 1+ iy + \frac{1}{2}(iy)^2+ \frac{1}{4}(iy)^3, \\
    |\Phi_r(iy)|^2 &= 1 - \frac{1}{4}y^4 +\frac{1}{16}y^6.
\end{align}

Since $|\Phi_r(iy)|^2- 1 \sim cy^4$ near $y = 0$,  we have $M = 4$. Since $M = 4$, $N = 3$ and thus $N - M + 1 = 0$, the exceptional case falls outside the scope of Theorem~\ref{theorem2}.

More generally, when $N=p+1$, we have $N-M+1\le 1$ since $M\ge p+1$. Note that because $\rho_r(y) = \Phi_r(iy)\Phi_r(-iy)$ is even, $M$ is even. Thus, when $p$ is even, $M\ge p+2$, and consequently $N-M+1\le 0$.

 Consider the case $(p,N)=(4,5)$ for which the optimal real polynomial is 
\begin{align}
    \Phi_r(iy) &= 1+ iy + \frac{1}{2}(iy)^2+ \frac{1}{6}(iy)^3+ \frac{1}{24}(iy)^4+ \frac{1}{144}(iy)^5, \\
    |\Phi_r(iy)|^2 &=1-\frac1{1728}y^8+\frac1{20736}y^{10}.
\end{align}

Since $M =8$, we have $N-M+1=-2$ and this case lies well outside the scope of Theorem~\ref{theorem2}. Nevertheless, Tables \ref{tab:stability_intervals} and \ref{tab:coefficients} show that
the optimal stability polynomial is  complex-valued.
\begin{align}
    \Phi_\varepsilon(iy) &= 1+ iy + \frac{1}{2}(iy)^2+ \frac{1}{6}(iy)^3+ \frac{1}{24}(iy)^4+ (0.00603+0.00359i)(iy)^5, \\
    |\Phi_\varepsilon(iy)|^2 &= 1  - 0.00718 y^5  +O(y^6).
\end{align}
Although the complex perturbation introduces a term of degree five, which is lower than $M=8$, its coefficient is negative and we obtain the necessary condition $\rho_\varepsilon(y) \leq 1$ as $y \to 0$. This reinforces Remark~\ref{remark:theorem2} by showing that the condition $N-M+1>0$ is sufficient, but not necessary, for the existence of an enlarging complex perturbation.

\subsection{Numerical example: The nonlinear Schrödinger equation}

To demonstrate the practical efficiency of using the constructed complex time integrators with expanded stability regions, we consider the nonlinear Schrödinger equation.
\begin{align}
    i u_t +\frac{1}{2}u_{xx} + |u|^2  u  = 0, \quad x \in [-2\pi, 4 \pi], \quad t\in [0, 6].
\label{eq:nonlinearsch}
\end{align}
The initial condition $u(x,0) = u_0$ is chosen to admit the following soliton exact solution.
\begin{align}
    u(x,t) = \sqrt{2}\sech(\sqrt{2}(x-t))e^{i(x+0.5t)}.
\end{align}
Spatial differentiation is done using the Fast Fourier Transform (FFT) with 100 modes \cite{trefethen2000spectral, fornberg1998practical}. Periodic boundary conditions are imposed on the domain. This leads to a discretized system similar to the form discussed in the previous section.

We simulate the nonlinear Schrödinger equation \eqref{eq:nonlinearsch} with the optimally stable real and complex integrators  (Equations \eqref{eq:stabschro_real} and \eqref{eq:stabschro_complex}) obtained in the previous subsection (see also Figure \ref{4_4equation} again). The optimally stable real integrator advances the solution from $u^n$ in two stages. Similar to the midpoint method, an intermediate value is first computed using a forward Euler step,
\begin{align}
u^{(1)} = u^n + \Delta t\,F(u^n),    
\end{align}
and then the solution after a time step $\Delta t$ is updated as 
\begin{align}
    u^{n+1} = u^n + \Delta t\,F\bigl(u^{(1)}\bigr),
\end{align}
where $F(u)$ denotes the right-hand side of the nonlinear Schrödinger equation after spatial discretization. The resulting stability polynomial is given by Equation \eqref{eq:stabschro_real}.

The optimally stable complex integrator follows a similar two-stage structure but employs a complex intermediate step.
\begin{align}
u^{(1)} &= u^n + \frac{(1-i)}{2}\Delta t\,F(u^n),  \\
    u^{n+1} &= u^n + \Delta t\,F\bigl(u^{(1)}\bigr).
\end{align}
The resulting stability polynomial is given by Equation \eqref{eq:stabschro_complex}.

Figure \ref{4_5equation} shows the error and computational time associated with various time steps. The optimal complex integrator (green) allows for a stable step size twice as large with even smaller error when compared to the optimal real integrator (magenta). Note that there is minimal additional computational effort for the complex integrator, since the Schrödinger equation already requires complex arithmetic even for a real integrator. Since both the real and complex integrators require nearly identical operations per time step (two function evaluations involving complex arithmetic), a $2\times$ larger stable time step directly reduces the required number of steps by half, resulting in an overall $2\times$ reduction in computational cost.

\begin{figure}

    \includegraphics{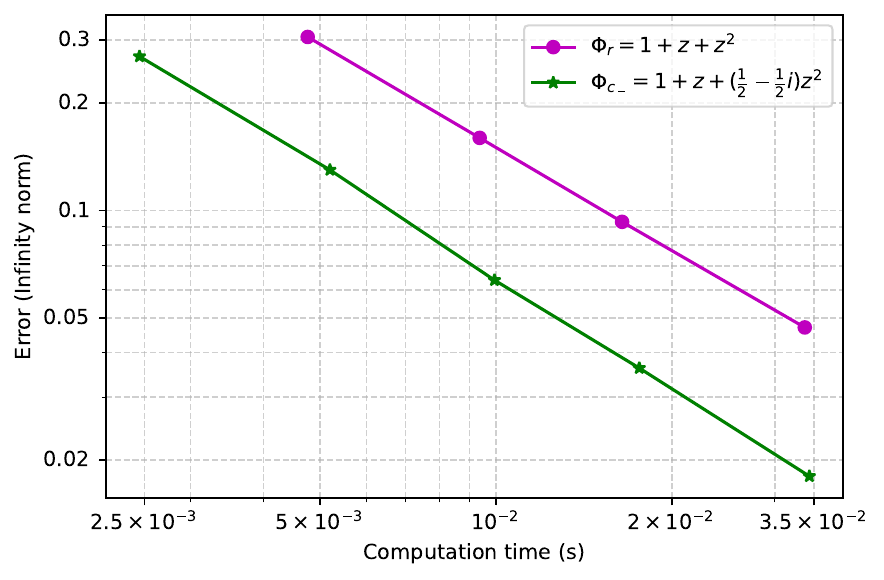}
    \caption{Simulation of the nonlinear Schrödinger equation \eqref{eq:nonlinearsch} using the optimally stable 2-stage first-order real (magenta, Equation \eqref{eq:stabschro_real}) and complex (green, Equation \eqref{eq:stabschro_complex}) integrators for $\Delta t = 0.014, 0.007, 0.0035, 0.002, 0.001$. The largest time step $\Delta t = 0.014$ is unstable and not shown for the real integrator. Errors and computational times are averaged across 1000 runs. The complex integrator allows for twice as large stable time steps and  half as much computational effort for slightly lower error.}
\label{4_5equation}

\end{figure}

\subsection{Numerical example: Time-dependent Density Functional Theory, The Kohn-Sham equation}

We now consider a real-world application of complex time integrators: Time-Dependent Density Functional Theory (TDDFT) \cite{Martin_2004, kohn-sham}. While the time-dependent Schr\"odinger equation exactly governs the evolution of a many-electron wavefunction, solving it directly or even storing it, is computationally intractable beyond a handful of electrons. In practice, one instead solves the Kohn--Sham equation~\eqref{eq:tdks}, which forms the basis of TDDFT \cite{kohn-sham,kohn1999nobel,Martin_2004, KSSaad},

\begin{equation}
i\,\frac{\partial \phi(\mathbf{r},t)}{\partial t}
= H_{KS}[\rho](\mathbf{r},t)\,\phi(\mathbf{r},t),
\qquad \mathbf{r}\in\Omega\subset\mathbb{R}^3,
\label{eq:tdks}
\end{equation}
where $\phi(\mathbf{r},t)$ is the Kohn--Sham orbital, $\Omega$ is the (finite, non-periodic) simulation domain and $H_{KS}[\rho]$ is the Kohn--Sham Hamiltonian, a functional of the electron density $\rho(\mathbf{r},t)$.

This recasts the exact many-body dynamics as a
single-particle-like evolution of  $\phi(\mathbf{r},t)$  coupled self-consistently through $\rho$. Equation~\eqref{eq:tdks} falls into the class of stiff complex-valued initial value problems the integrators in Table \ref{tab:stability_intervals} are designed for, and
it is the application most frequently invoked in the TDDFT numerics literature \cite{kang2019pushing, kononov2022electron, an2020quantum}.

As a concrete three-dimensional test case,  we consider the
$\mathrm{H}_2$ molecule following a weak delta-kick perturbation of its ground state~\cite{YabanaBertsch1996}. $\mathrm{H}_2$ has two electrons in a closed-shell configuration. The density $\rho = 2|\phi|^2$ is carried by a single doubly-occupied orbital and is an exact description of the all-electron system with no core
states to remove and no pseudopotential to tune. The Kohn--Sham Hamiltonian,
\begin{equation}
H_{KS}[\rho] = -\tfrac{1}{2}\nabla^2 + V_\text{ext}(\mathbf{r})
+ V_H[\rho](\mathbf{r}) + V_{xc}[\rho](\mathbf{r}),
\label{eq:hks}
\end{equation}
uses the true two-proton external potential,
\begin{equation}
V_\text{ext}(\mathbf{r}) = -\frac{1}{\sqrt{|\mathbf{r}-R_{\text{H--H}}\hat{z}/2|^2 + a^2}}-\frac{1}{\sqrt{|\mathbf{r}+R_{\text{H--H}}\hat{z}/2|^2 + a^2}},
\label{eq:vext}
\end{equation}
regularized by a soft-Coulomb parameter $a = 0.15\,a_0$ that removes the
$1/r$ singularity, where $R_{\text{H--H}} = 0.74\,a_0$ is the bond length in terms of Bohr length $(a_0)$ (atomic unit of length). 

The Hartree potential $V_H$ is the classical electrostatic potential generated by $\rho$ and satisfies the Poisson equation,
\begin{equation}
\nabla^2 V_H(\mathbf{r}) = -4\pi\,\rho(\mathbf{r}),
\label{eq:vh}
\end{equation}
 with the boundary condition $V_H(\mathbf{r}) = 2/r$ on $\partial\Omega$,
appropriate for the total charge $N_e = 2$ enclosed. Exchange-correlation effects are treated using the  local density approximation,
$V_{xc} = V_x + V_c$, with Slater exchange \cite{slater}
\begin{equation}
V_x(\rho) = -\Big(\frac{3\rho}{\pi}\Big)^{1/3}
\label{eq:vx}
\end{equation}
and Perdew--Zunger correlation~\cite{perdew_zunger}, $V_c$, 
\begin{equation}
V_c(r_s) =
\begin{cases}
\epsilon_c(r_s)\,\dfrac{1+\tfrac{7}{6}\beta_1\sqrt{r_s}+\tfrac{4}{3}\beta_2 r_s}
{1+\beta_1\sqrt{r_s}+\beta_2 r_s}, & r_s \ge 1,\\[10pt]
A\ln r_s + \big(B-\tfrac{A}{3}\big) + \tfrac{2}{3}C\,r_s\ln r_s
+ \tfrac{2D-C}{3}\,r_s, & r_s < 1,
\end{cases}
\label{eq:vc}
\end{equation}
with the paramagnetic ($\zeta = 0$) parameters $\gamma = -0.1423$,
$\beta_1 = 1.0529$, $\beta_2 = 0.3334$, $A = 0.0311$, $B = -0.0480$,
$C = 0.0020$, $D = -0.0116$~\cite{perdew_zunger}. 

We discretize the Kohn--Sham equation on a finite, non-periodic cubic domain $\Omega = [-L,L]^3$ with $L = 7 a_0$ with a uniform finite-difference grid of spacing $h = 0.125 a_0$. The Laplacian $\nabla^2$ is approximated by the second-order, seven-point finite-difference stencil. The Poisson equation \eqref{eq:vh} for the Hartree potential is solved using a discrete sine transform (DST), which diagonalizes the discrete Laplacian under the Dirichlet boundary conditions \cite{KSSaad}.

The discretized $H_{KS}$ is Hermitian and its eigenvalues, the Kohn-Sham orbital energies, are bounded from below with $\varepsilon_j \ge \varepsilon_\text{min} = \varepsilon_\text{HOMO}$,  the energy of the highest occupied molecular orbital (HOMO). The initial condition uses the ground-state orbital $\phi^{\text{GS}}$ which is the solution to the self-consistent Kohn--Sham eigenvalue problem,
\begin{equation}
H_{KS}[\rho]\,\phi^{\text{GS}} = \varepsilon_\text{HOMO}\,\phi^{\text{GS}},
\qquad \rho = 2|\phi^{\text{GS}}|^2,
\label{eq:gs}
\end{equation}
obtained by damped fixed-point (linear-mixing) iteration on the
density~\cite{Martin_2004}.

Since total energy is defined only up to an additive constant, a shifted Hamiltonian \cite{shepard, jia2018} is used in the time propagation,
\begin{equation}
H_{KS}' = H_{KS} - \varepsilon_\text{HOMO}\,\mathbb{I},
\label{eq:shift}
\end{equation}
so that the occupied level sits at zero. This amounts to a global phase change, $\phi'(t) = e^{i\varepsilon_\text{HOMO}t}\phi(t)$, leaving $\rho = 2|\phi|^2$ and all observables unchanged. We use $H_{KS}'$ in the numerical experiments below. The dynamics are initiated from  $\phi^{\text{GS}}$ by a uniform momentum boost along the $\hat{z}$ axis
~\cite{YabanaBertsch1996},
\begin{equation}
\phi(\mathbf{r},0^+) = e^{\,i\kappa(\hat{z}\cdot\mathbf{r})}\,
\phi^{\text{GS}}(\mathbf{r}),
\label{eq:kick}
\end{equation}
with kick strength $\kappa = 10^{-3}\,a_0^{-1}$ small enough to keep the response linear and propagate it for time $t = 1 $ a.u. (atomic units).

We compare the  performance of the optimally stable real and complex 4-stage, first-order integrators, $\Phi_r^{(4,1)}$ and $\Phi_{c_-}^{(4,1)}$ corresponding to the stability intervals in Table~\ref{tab:stability_intervals}. For any $(N,p)$ entry of Table~\ref{tab:coefficients}, an integrator with the stability polynomial $\Phi(z)=1+z+\sum_{k=2}^{N}a_kz^k$ of that row can be realized using the following update rule on $y^n$, the numerical solution at time $t = n\Delta t$.

\begin{align}
Y_0 = y^n, \qquad
Y_k = y^n + c_k\,\Delta t\,F\bigl(Y_{k-1}\bigr), \qquad
k=1,\dots,N, \qquad y^{n+1}=Y_N,
\end{align}
with 
\begin{align}
    c_N=1, \qquad c_j = \frac{a_{N-j+1}}{a_{N-j}}, \qquad j=1,\dots,N-1, \qquad a_1 = 1. 
\end{align}

For $\Phi_{c_-}^{(4,1)}$, the coefficients $a_k$ are the complex conjugates
of the $(4,1)$ entry of Table~\ref{tab:coefficients}, giving 
\begin{align}
    c_1=\frac{a_4}{a_3}, \qquad c_2=\frac{a_3}{a_2}, \qquad c_3=a_2, \qquad c_4=1.
\end{align} The real coefficients used for $\Phi_r^{(4,1)}$ are 
\begin{align}
    c_1 = \frac13, \qquad  c_2 = \frac4{15}, \qquad  c_3 =\frac59, \qquad  c_4=1,
\end{align} and the corresponding stability polynomial is 
\begin{align}
    1 + z + \frac59 z^2 + \frac4{27} z^3 + \frac4{81} z^4.
\end{align} We also include the optimal 2-stage real and complex integrators (Equations \eqref{eq:stabschro_real} and \eqref{eq:stabschro_complex}) used in the previous subsection in our comparison.

Accuracy is measured by the final-time orbital error $
\left\lVert\phi(T_f)-\phi^{\mathrm{ref}}(T_f)\right\rVert_{L^2}$, where $\phi(T_f)$ is the Kohn--Sham orbital obtained using the tested integrator and $\phi^{\mathrm{ref}}(T_f)$ is a fourth-order Runge--Kutta (RK4) reference solution obtained with a much smaller time step. The $L^2$ norm is evaluated over the three-dimensional computational grid at the final time $T_f$.

\begin{figure}
\begin{centering}
    \includegraphics[width= \textwidth]{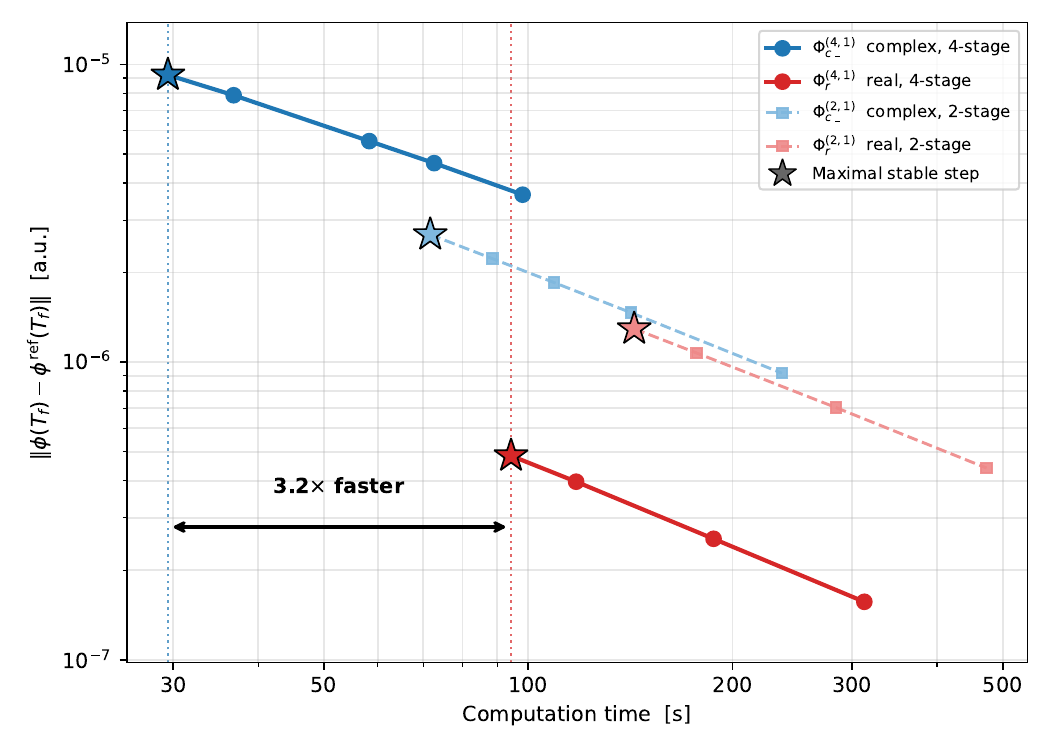}
    \caption{Final-time orbital error ($L^2$ norm) versus computation time of the simulation for the three-dimensional time-dependent Kohn--Sham dynamics of the $H_2$ molecule following a perturbation of its ground state. Results are shown for four integrators, $\Phi_r^{(2,1)}$, $\Phi_{c_-}^{(2,1)}$, $\Phi_r^{(4,1)}$, and $\Phi_{c_-}^{(4,1)}$, corresponding to the optimal stability intervals in Table~\ref{tab:stability_intervals} with $N=2,4$ and $p=1$. Each integrator is evaluated at step sizes of $0.3$, $0.5$, $0.8$, and $0.99$ times its maximum stable step size. At their maximum stable step sizes, the complex
integrators reduce wall-clock time by factors of $2$ and $3.2$ relative to their real counterparts for $N=2$ and $N=4$, respectively.}
\label{fig:TDKS}
\end{centering}
\end{figure}

Figure~\ref{fig:TDKS} shows that while $\Phi_r^{(4,1)}$ achieves the lowest error for a fixed step size among the four integrators, $\Phi_{c_-}^{(4,1)}$ admits a much larger stable step and is preferable  if computational cost is the main priority. The dynamics in this subsection were simulated for $1$ a.u. which is sufficient to observe numerical instability when the step size exceeds the stability limit (e.g., at $1.1\times$ the maximal stable step estimated from the largest eigenvalue of $H_{KS}'[\phi^{\mathrm{GS}}]$ ) for all integrators. In practical applications, however, TDDFT simulations often span $500$--$2000$ a.u.   \cite{tddftlopata,shepard,excitedstate}, making low-runtime integrators especially important. 

Note that for the optimal real integrators with maximal stable step size  along the negative imaginary axis (Table 2 in \cite{ketcheson2013optimal}), the normalized stability interval $\Delta t_{\max}/N_{\text{stage}}$ appears to approach $1$ from below as $N_{\text{stage}}\to\infty$.  For the corresponding optimal complex integrators, $\Delta t_{\max}/N_{\text{stage}}$ equals $1$ already at $N_{\text{stage}}=2$ and continues to increase thereafter, with no apparent upper bound over the range of stages considered. This trend suggests that very high-stage complex integrators may offer practical computational advantages for long-time simulations of the Kohn--Sham equations, since their normalized stability interval continues to grow with the number of stages.

We show through these examples that complex steps/coefficients in a time integrator can expand stability regions beyond real equivalents and allow for larger time steps. Even larger regions may be possible using methods with more stages/sub-steps, but the cost trade-off between increased evaluations and increased step size would need to be evaluated on a problem-specific basis. In particular, they suggest a pathway to more efficient simulations of the complex-valued Kohn-Sham equation, and potentially other quantum systems, by exploiting asymmetric stability regions to allow larger stable time steps.

\section{The Projective Integration method with complex time steps} \label{sec:PI}

As discussed in the previous sections, sacrificing formal order of accuracy in exchange for a larger stability region can substantially improve computational efficiency. Projective Integration (PI) \cite{maclean2015convergence,pi1,pi5,pi4} exploits this idea for stiff systems with a spectral gap. It damps the fast modes with a few small inner steps and then extrapolates over a much larger outer step, thereby avoiding the tiny time steps required by conventional explicit methods. In its standard form, however, the inner integrator is the Forward Euler method, whose stability region is confined near the negative real axis. Consequently, PI is effective only when the fast eigenvalues are nearly real and loses stability when they acquire significant imaginary components. By replacing the inner Forward Euler steps with complex time steps, we shift the inner stability region to enclose complex fast modes while preserving the large outer projective step and, hence, the efficiency of PI.

We first review the PI method, then introduce its complex extension on a scalar stiff problem with complex eigenvalues, and finally demonstrate its performance on a real-valued multiscale system whose fast dynamics are governed by a complex-conjugate pair of eigenvalues.

\subsection{Projective Integration}
Projective Integration (PI) schemes are explicit schemes tailored to the stable integration of stiff ODE systems with spectral gaps. This means that they target systems of the form $\dot{y} = A y$, where $y\in \mathbb{R}^n$ and $A\in \mathbb{R}^{n \times n}$ is a matrix with two distinct clusters of eigenvalues $\sigma_f$ and $\sigma_s$, representing fast and slow modes, respectively. The slow modes $\lambda_s \in \sigma_s$ evolve on the time scale $\mathcal{O}(1)$ and are most relevant for accuracy, but this requires the fast modes $\lambda_f \in \sigma_f$ evolving on the short time scale $\mathcal{O}(\epsilon)$ to be integrated in a stable manner. For traditional, explicit schemes like the FE scheme, this usually results in a prohibitively small time step $\Delta t \approx \epsilon$. PI methods overcome this by damping out the fast modes in a series of small (sub-)steps of size $\delta t \approx \epsilon$ and then extrapolating the result over the remainder of a longer time step $\Delta t$, see \cite{pi1}.

A PI scheme in its simplest form utilizes a small number of $K+1$ small forward Euler time steps of size $\delta t \approx \epsilon$ followed by a larger extrapolation step of the remaining time step $\Delta t - (K+1)\delta t$. This is commonly referred to as the Projective Forward Euler (PFE) method, which can be written using $y^{n,0}=y^n$ and
\begin{eqnarray}
  y^{n,1} &=& y^{n,0} + \delta t \cdot f\left(y^{n,0}\right) \\
  y^{n,2} &=& y^{n,1} + \delta t \cdot f\left(y^{n,1}\right) \\
   &\vdots&  \\
  y^{n,K+1} &=& y^{n,K} + \delta t \cdot f\left(y^{n,K}\right) \\
  y^{n+1} &=& y^{n,K+1} + (\Delta t - (K+1)\delta t) \cdot \frac{y^{n,K+1}-y^{n,K}}{\delta t}.
  \label{e:PI_def}
\end{eqnarray}
In \cite{KOELLERMEIER2025116147} it was recently shown that the PFE method can be written as a standard RK scheme with $K+1$ stages using $\Lambda := \frac{\delta t}{\Delta t}$ and the Butcher tableau given by \eqref{tab:Butcher_PFE_as_RK}.
\begin{equation}
\begin{array}{ c | c c c c }
 0 & 0 & & & \\
 1 \cdot \Lambda & \Lambda & 0 & & \\
 \vdots & \vdots & \ddots & \ddots & \\
 K \cdot \Lambda & \Lambda & \dots & \Lambda & 0 \\
 \hline
  & \Lambda & \dots & \Lambda & 1 - K \Lambda
  \label{tab:Butcher_PFE_as_RK}
\end{array}
\end{equation}
The stability polynomial of a PFE scheme with $K+1$ inner steps is given by (compare \cite{lafitte_high-order_2016})
\begin{align}
    \Phi(z) =
(1+\Lambda z)^{K}
\left( 1 + (1-K\Lambda) z \right).
\end{align}

Up to now, PI schemes have been used exclusively with real time step sizes $\Delta t$ and $\delta t$ and FE as inner integrator. This limited their applicability to problems where the eigenvalues have small imaginary parts, as the commonly used FE inner integrators are typically not able to capture eigenvalues with significant imaginary parts. Allowing complex time steps opens up the possibility of larger stable time steps, extending PI methods to  applications with larger imaginary eigenvalues. This is a first direction to using more advanced inner integrators and the use of other high-order inner integrators (with complex-valued as well as real-valued time step sizes) is left for future work.

In the remainder of this section, we will apply complex time stepping to PI schemes and use this to solve both real-valued and complex-valued systems.

\subsection{Complex PI schemes for a complex-valued stiff equation}
Real Forward Euler steps are typically used as the inner steps in PI schemes, while either first-order FE or higher order RK schemes are utilized for the outer steps, see \cite{lafitte_high-order_2016} for high order PI schemes. If $\Phi_{outer}$ and $\Phi_{inner}$ denote the stability polynomials corresponding to the outer step and the inner step(s), respectively, the full stability polynomial $\Phi_{PI}$ for a PI scheme can be expressed as a concatenation 
\begin{align}
    \Phi_{PI}(z) = \Phi_{inner}(z) \cdot \Phi_{outer}(z).
\end{align}

If the inner step(s) consist of $K+1$ Forward Euler time steps with variable time step sizes $\delta t_n = a_n \Delta t$, then $\Phi_{inner}$ can be written as a concatenation of these substeps.
\begin{align}
    \Phi_{inner}(z) = \prod_{n=1}^{{K}} (1 + a_n z).
\end{align}

The roots of $\Phi_{inner}(z)$ are located at $- 1/a_n$. For real time steps ($a_n \in \mathbb{R}$), this results in stability regions that are centered around the negative real axis. For stiff problems where the eigenvalues contain a significant imaginary component, such stability regions centered on the real axis can lead to instability. 

To illustrate this, consider the Prothero-Robinson example, a classic stiff differential equation \cite{leveque2007finite}, in Equation \eqref{eq:classic_stiff}
\begin{align}
\frac{dy}{dt} = \lambda (y - \cos{t}) - \sin{t}, \quad y(0) = \frac{3}{2}.
    \label{eq:classic_stiff}
\end{align}
For large values of $\lambda$, Equation \eqref{eq:classic_stiff} has fast transient dynamics away from its steady state solution $y = \cos (t)$. Although the scalar Equation \eqref{eq:classic_stiff} does not contain both slow and fast dynamics, it suffices as a simple example to outline the ability of complex PI schemes to capture fast dynamics for problems with non-negligible imaginary eigenvalue parts.

We now consider the complex-valued case where $\lambda \in \mathbb{C}$ can be represented as $\lambda = -\frac{1}{\varepsilon} + \xi i$ where $\varepsilon \in \mathbb{R}_+$, $\xi \in \mathbb{R}$ and $\xi i$ is a small imaginary perturbation around the dominant stiff component $-\frac{1}{\varepsilon}$.

We solve the Prothero-Robinson equation \eqref{eq:classic_stiff} numerically  using two different PI methods: 
\begin{enumerate}[label=(\Alph*)]
    \item \textbf{Real Projective Forward Euler (PFE):} two real inner steps 
    \(\delta t =-1/\mathrm{Re}(\lambda)\) 
    and the outer extrapolation step for the remainder of  
    \(\Delta t = 0.05\).
    
    \item \textbf{Complex Projective Forward Euler (cPFE):} two complex inner steps 
    \(\delta t = -1/\lambda\) 
    and the outer extrapolation step for the remainder of 
    \(\Delta t = 0.05 \).
\end{enumerate}
For both schemes, given the numerical solution $y^n$
at time $t$, the solution $y^{n+1}$ at time $t + \Delta t$ is computed via (compare \eqref{e:PI_def})
\begin{eqnarray}
  y^{n,1} &=& y^{n} + \delta t \cdot f\left(y^{n}\right), \label{eqn:PI_trad_0}\\
  y^{n,2} &=& y^{n,1} + \delta t \cdot f\left(y^{n,1}\right), \\
  y^{n+1} &=& y^{n,2} + (\Delta t - 2 \delta t) \cdot \frac{y^{n,2}-y^{n,1}}{\delta t},
  \label{eqn:PI_trad}
\end{eqnarray}
where $f(y)$ is the right hand side of Equation \eqref{eq:classic_stiff}.
Note that  \eqref{eqn:PI_trad_0} -\eqref{eqn:PI_trad} is equivalent to the following two-step Forward Euler update:
\begin{eqnarray}
  y^{n,1} &=& y^{n} + \delta t \cdot f\left(y^{n}\right) \label{PI_newa}\\  
  y^{n+1} &=& y^{n,1} + (\Delta t - \delta t) \cdot f\left(y^{n,1}\right).
  \label{PI_new}
\end{eqnarray}

The stability diagrams for the real PFE and the complex cPFE described above as (A) and (B), respectively, are illustrated in Figure \ref{fig:PI_complex_stab}. We can see that the stability region of the cPFE method encompasses the complex eigenvalue $\lambda$ while the real PFE method fails to do so. 
\begin{figure}
\begin{centering}
\includegraphics{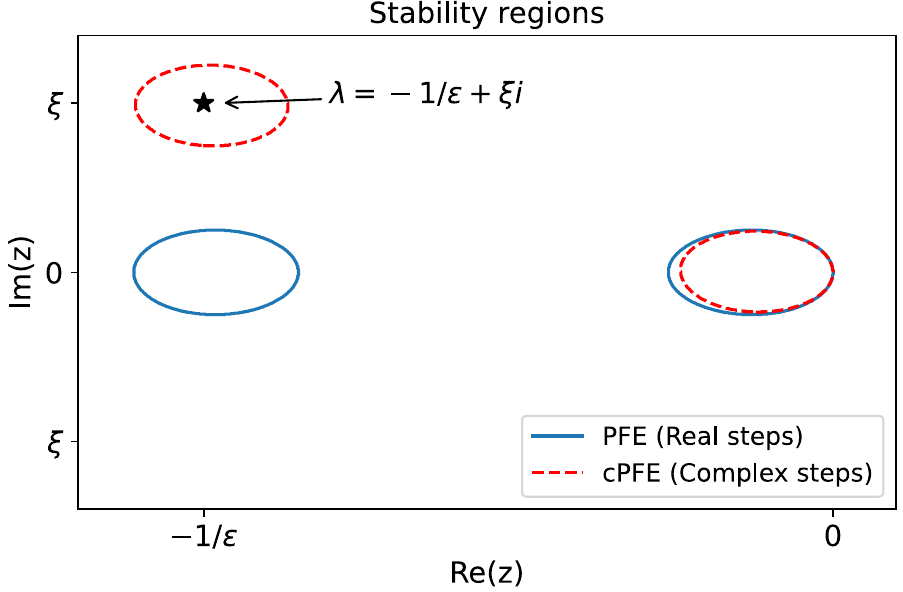}
\caption{Stability regions corresponding to the real PFE method (blue) and complex cPFE method (red), together with the location of the complex eigenvalue~$\lambda$ in \eqref{eq:classic_stiff}.}
\label{fig:PI_complex_stab}
\end{centering}
\end{figure}
We present numerical results in Figure \ref{fig:PI_complex_num} for two representative cases:

\begin{enumerate}
    \item \(\lambda = -10^6 + 15i\), i.e., $\varepsilon = 10^{-6}, \xi = 15$, corresponding to a small imaginary part; and
    \item \(\lambda = -10^6 + 20i\), i.e., $\varepsilon = 10^{-6}, \xi = 20$, corresponding to a larger imaginary part.
\end{enumerate}
In both cases, the real part of $\lambda$ is much larger in magnitude than the imaginary part. 

For $\lambda = -10^6 + 15i$ in Figure \ref{fig:PI_complex_num} (left), the real PFE solution is oscillating but the oscillations are damped in time and the method is therefore stable. However, for only a relatively small increase in the imaginary part from $\lambda = -10^6 + 15i$ to  $\lambda = -10^6 + 20i$ in Figure \ref{fig:PI_complex_num} (right), the real PFE solution has persisting oscillations that are only semi-stable. The complex cPFE solution, however, is stable in both cases. This clearly demonstrates the benefits of complex time integration for this first example of PI schemes. However, since this example lacks a spectral gap, we next turn to a more challenging system exhibiting coupled slow and fast dynamics.

\begin{figure}
\begin{centering}
\includegraphics{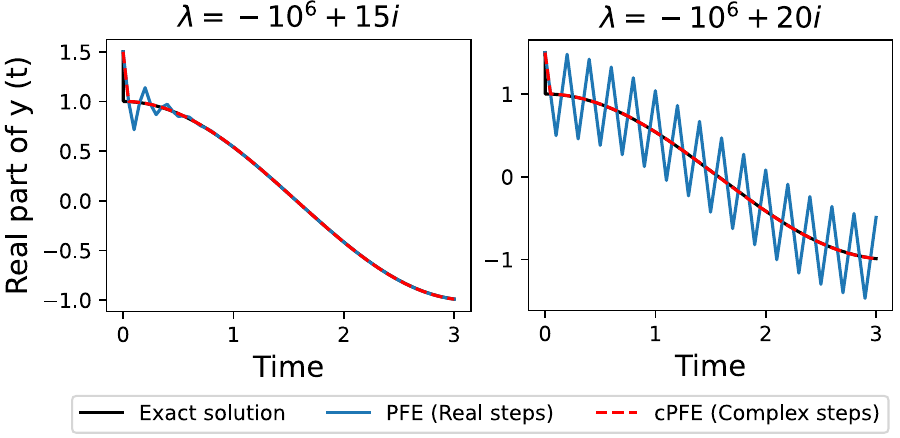}
\caption{Numerical solutions of the Prothero–Robinson example~\eqref{eq:classic_stiff} for a small imaginary component, $\lambda = -10^6 + 15i$ (left), and a larger imaginary component, $\lambda = -10^6 + 20i$ (right).
The exact solution is shown in black, the real PFE in blue, and the complex cPFE in dashed red.
The complex cPFE method remains stable and free of oscillations in both cases due to its stability region being shifted off the real axis, allowing it to capture the complex eigenvalue~$\lambda$.
}
\label{fig:PI_complex_num}
\end{centering}
\end{figure}

\subsection{Complex PI schemes for real-valued systems}
After the stiff scalar but complex-valued example, we now consider the two-scale, potentially stiff but real-valued system
\begin{align}
    \dot{y_1}&=y_2, \label{eq:stiff_system1}\\
    \dot{y_2}&=-2\mu y_2-\omega_0^2 y_1, \\
    \dot{y_3}&=-\lambda(y_3-y_1). \label{eq:stiff_system3}
\end{align}
The system \eqref{eq:stiff_system1}-\eqref{eq:stiff_system3} models the slow decay of a variable $y_3$ to zero with decay rate $\lambda \in \mathbb{R}_+$, perturbed by the solution $y_1$ of a damped oscillator with natural frequency  $\omega_0 \in \mathbb{R}_+$ and viscous damping coefficient $2\mu \in \mathbb{R}_+$. For a heavily damped system, we can write $\mu$ as $\mu = 1/\varepsilon$, where $ 0 < \varepsilon \ll 1$. We further assume a highly oscillatory system where $\omega_0^2 = 1/\varepsilon^2 +\delta^2$, for some $\delta > 0$. Here, we are interested in computing a stable solution for $y_3$ without necessarily resolving the fast dynamics of $y_1$.

System \eqref{eq:stiff_system1}-\eqref{eq:stiff_system3} can be written as the linear system
\begin{align}
    \left(\begin{array}{l}
    \dot{y_1} \\ \dot{y_2} \\ \dot{y_3}
    \end{array}\right)=\underbrace{\left(\begin{array}{ccc}
    0 & 1 & 0 \\
    -1/\varepsilon^2 -\delta^2 & -2/\varepsilon & 0 \\
    \lambda & 0 & -\lambda
    \end{array}\right)}_A \left(\begin{array}{l}
    y_1 \\ y_2 \\ y_3
    \end{array}\right),
    \label{eqn:linearsystem}
\end{align}
where the system matrix $A \in \mathbb{R}^{3 \times 3}$ has eigenvalues
\begin{align}
\lambda \in \sigma(A) = \left\{-\lambda, -\frac{1}{\varepsilon}+i\delta,-\frac{1}{\varepsilon}-i\delta\right\},
\end{align}
thus resulting in two complex conjugate eigenvalues. 

For $\lambda \ll  1/\varepsilon$, the system is stiff, with scale separation between the fast oscillatory movement of $y_1$ and the slow decay of $y_3$ as illustrated in Figure \ref{fig:PI_complex2} (right). We consider the parameters $\lambda = 1, \varepsilon = 0.1, \delta = 22.913$ with the initial conditions $(y_1,y_2,y_3) = (4, 2, 5)$. The eigenvalues are shown as black dots in Figure \ref{fig:PI_complex2} (left), together with the stability region of the real PI method (blue),  and the stability region of the complex PI method (red). Both methods use two inner evaluations of the right-hand side (Equation \eqref{eqn:linearsystem}) with inner step size $\delta t$ in order to implement Equation \eqref{PI_newa} . For the real PI method, both inner steps are $\delta t = -1/(-\frac{1}{\varepsilon}) =\varepsilon $, while for the complex PI method, the first inner step is  $\delta t_1 = -1/(-\frac{1}{\varepsilon} + i\delta )$ and the second is $\delta t_2 = -1/(-\frac{1}{\varepsilon} - i\delta )$. Note $K = 2$ for both methods. The real PI stability region includes only a narrow region around the imaginary axis while the complex PI stability region encompasses the complex eigenvalues. This means that a real PI method would lead to unstable numerical simulations for this application, while the complex PI method is stable by construction.
\begin{figure}
\includegraphics[width= \textwidth]{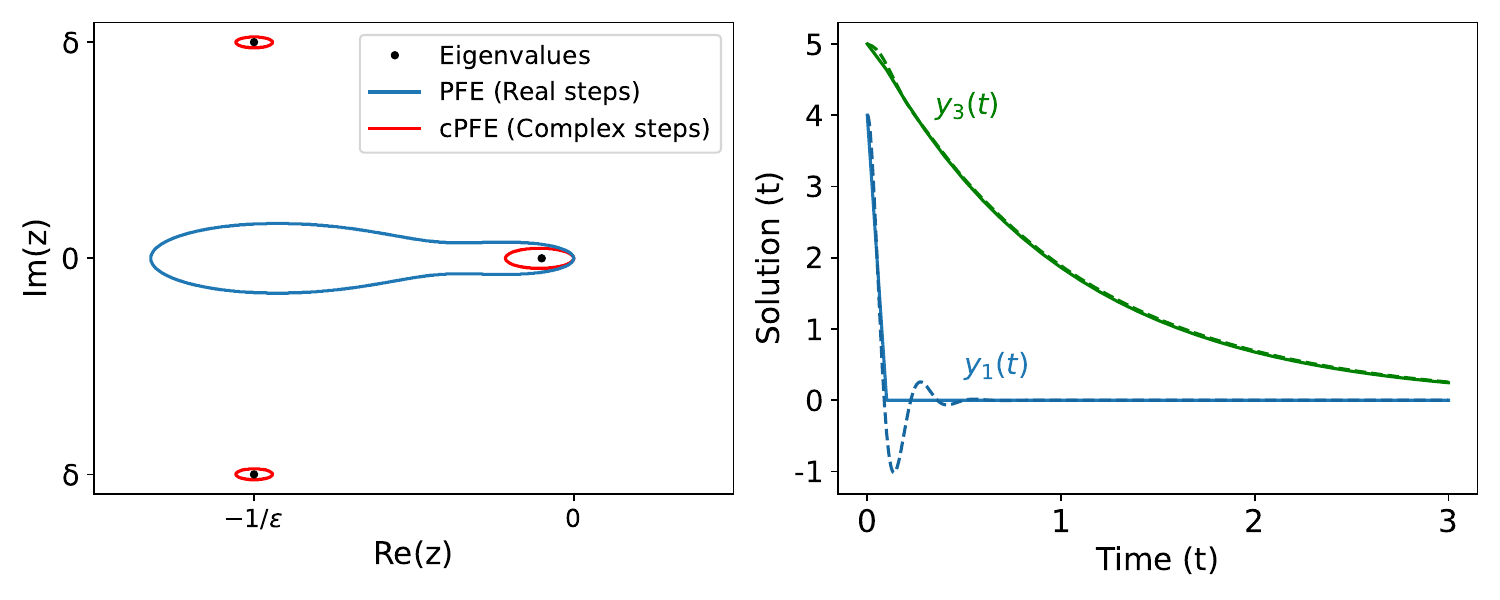}
    \caption{Left: Spectrum of the system \eqref{eqn:linearsystem}, consisting of one real eigenvalue associated with slow decay and a pair of complex conjugate eigenvalues modeling fast oscillatory dynamics. The stability region of the complex cPFE method (red) encloses all eigenvalues, while the real PFE stability region (blue) fails to capture the oscillatory modes. Right: Time evolution of \eqref{eq:stiff_system1}–\eqref{eq:stiff_system3} for parameters $\lambda = 1, \varepsilon = 0.1, \delta = 22.913$ with the initial conditions $(y_1,y_2,y_3) = (4, 2, 5)$.  $y_3$ is influenced by the oscillations of $y_1$ in the beginning, but $y_1$ quickly gets damped to 0, then $y_3$ slowly decays further. Dashed lines indicate the exact solution, while solid lines show the numerical solution obtained using the stable cPFE method (the real PFE method is unstable at this step size). The outer step is $\Delta t = 0.1$ and $K = 2$ with the first complex inner step $\delta t_1 = -1/(-\frac{1}{\varepsilon} + i\delta )$  and  the second inner step $\delta t_2 = -1/(-\frac{1}{\varepsilon} - i\delta )$. } 
  
\label{fig:PI_complex2}
\end{figure}

This example has been constructed to reflect the regime in which PI is intended to be used. The system exhibits a clear spectral gap, with a slow mode associated with the real eigenvalue  $-\lambda$ and fast modes associated with eigenvalues of large negative real part and non-negligible imaginary component. In PI, the outer time step is chosen according to the slow time scale of interest, while the role of the inner steps is to stabilize the fast dynamics. Consequently, we select $\Delta t \approx \mathcal{O}(\frac{1}{\lambda})$ corresponding to the slow decay rate, and do not adjust this outer step to accommodate stability restrictions imposed by the fast modes. At this slow-scale step size, a real PI scheme becomes unstable due to the oscillatory fast eigenvalues, whereas a complex PI scheme remains stable by appropriately shifting the stability region.

A stable real PI or real Forward Euler method is technically possible, but would require a much smaller time step, e.g.,  $\Delta t = 1/\|-\frac{1}{\varepsilon}+i\delta \|$ in case of the explicit Euler method. Such a choice would force the method to resolve the fast dynamics and negate the computational advantage that PI is designed to provide. In contrast, for a complex PI scheme, a large outer time step can simply be chosen as $\Delta t = \frac{1}{\lambda}$ and two complex conjugate inner time steps $\delta t_1 = -1/(-\frac{1}{\varepsilon} + i\delta )$  and  $\delta t_2 = -1/(-\frac{1}{\varepsilon} - i\delta )$ , allow to capture the complex eigenvalues. This results in a stable and efficient numerical solution as shown in Figure \ref{fig:PI_complex2} (right).

The examples in this section, a scalar complex-valued equation and a three-variable real system, are deliberately simple so that the inner and outer stability regions can be visualized explicitly and the effect of complex inner steps can be isolated. This simplicity does not limit the applicability of the method, which depends on the spectrum of the problem rather than its dimension. Any stiff system whose spectrum separates into slow modes of interest and a cluster of fast modes with non-negligible imaginary parts can be treated using the same construction, with the complex inner steps chosen to stabilize the fast cluster.

 The design of the inner steps, however, is problem-specific. In our examples, the fast eigenvalues are available exactly. At scale and in the presence of nonlinearities, they must be estimated \cite{pi1,KOELLERMEIER2025116147}. Fast oscillatory spectra of this kind are common in higher-dimensional applications, including chemical reaction networks with fast and slow dynamics \cite{Epstein1998}, electrical circuits \cite{Chen2013}, and control systems with complex feedback loops \cite{Astrom2021}. Extending complex PI to such systems, together with the spectral estimation and associated inner-step design, is the natural way to build on the mechanism established here.

\section{Discussion}
This paper explores the benefits of leaving the real line for the numerical time integration of stiff problems. We show analytically and numerically that complex time stepping can be used to expand stability regions, allowing larger, stable time steps. Expanded stability regions have been previously designed using Runge-Kutta methods \cite{riha1972optimal, ketcheson2013optimal}, but we show that complex coefficients in the optimal stability polynomial allow for even larger stability regions, particularly for systems with complex-valued spectra.

The fact that the stability region can be increased by changing the integration domain rather than the integration method may also be appealing in scenarios where particular integrators with specific properties (such as nonlinear stability) for specific differential equations are desired. Although our analysis is restricted to linear stability, the numerical results for the nonlinear Schrödinger equation suggest that complex time-stepping strategies can be effective in nonlinear settings as well. This observation points toward the possibility that complex integrators could satisfy nonlinear stability guarantees such as Strong Stability Preserving properties \cite{gottlieb2001strong,gottlieb2009high}, although establishing such guarantees lies beyond the scope of the present work.

We intend this paper to be an introduction to and initiation of the search for efficient time integrators that take advantage of paths in the complex plane. For simulating quantum systems with inherently complex-valued equations \cite{kononov2022electron}, complex integrators allow for expanded stability regions at minimal computational cost, making it a particularly promising application of our method. We have mostly focused on creating integrators with improved stability, but  many other desirable properties including energy and momentum preservation, remain to be explored. The complex plane offers an extra dimension on which any integrator can be improved, opening up a new realm of possibilities when designing time-stepping methods.

\section{Acknowledgements}

We would like to thank Emil Constantinescu, David Ketcheson, Alvin Bayliss and David Chopp for extensive feedback and valuable suggestions.

\bibliographystyle{unsrt}
\bibliography{references}

\end{document}